\documentclass{amsart}

\usepackage{amsmath,amsthm,amssymb,bm}
\usepackage{hyperref}
\usepackage{a4wide}
\usepackage{cleveref}
\usepackage{graphicx,color}
\usepackage{tikz}
\usetikzlibrary{decorations.markings}
\usetikzlibrary{decorations.pathreplacing}
\usetikzlibrary {arrows.meta}
\usepackage{mathtools}
\def\multiset#1#2{\ensuremath{\left(\kern-.4em\left(\genfrac{}{}{0pt}{}{#1}{#2}\right)\kern-.4em\right)}}
\newcommand*{\xMin}{0}%
\newcommand*{\xMax}{6}%
\newcommand*{\yMin}{0}%
\newcommand*{\yMax}{3}%

\numberwithin{equation}{section}
\newtheorem{thm}{Theorem}[section]

\theoremstyle{definition}

\newtheorem{problem}[thm]{Problem}
\newtheorem{remark}[thm]{Remark}

\newtheorem{theorem}[thm]{Theorem}
\newtheorem{lemma}[thm]{Lemma}
\newtheorem{proposition}[thm]{Proposition}
\newtheorem{corollary}[thm]{Corollary}
\newtheorem{definition}[thm]{Definition}

\AddToHook{env/lemma/begin}{\crefalias{thm}{lemma}}
\AddToHook{env/proposition/begin}{\crefalias{thm}{proposition}}
\AddToHook{env/corollary/begin}{\crefalias{thm}{corollary}}
\AddToHook{env/definition/begin}{\crefalias{thm}{definition}}
\AddToHook{env/remark/begin}{\crefalias{thm}{remark}}
\AddToHook{env/problem/begin}{\crefalias{thm}{problem}}
\crefname{lemma}{Lemma}{Lemmas}
\Crefname{lemma}{Lemma}{Lemmas}
\crefname{proposition}{Proposition}{Propositions}
\Crefname{proposition}{Proposition}{Propositions}
\crefname{corollary}{Corollary}{Corollaries}
\Crefname{corollary}{Corollary}{Corollaries}
\crefname{definition}{Definition}{Definitions}
\Crefname{definition}{Definition}{Definitions}
\crefname{remark}{Remark}{Remarks}
\Crefname{remark}{Remark}{Remarks}
\crefname{problem}{Problem}{Problems}
\Crefname{problem}{Problem}{Problems}

\newcommand{\ZZ}{\mathbb{Z}}

\newcommand\LL{\mathcal{L}}

\newcommand\wt{\operatorname{wt}}
\newcommand\gm{\mathsf{R}}
\newcommand\Mot{\mathsf{Mot}}
\newcommand\Sch{\mathsf{Sch}}
\newcommand\MS{\mathsf{MS}}
\newcommand\rgm{\widetilde{\mathsf{R}}}
\newcommand\dt{\operatorname{DT}}

\renewcommand\vec[1]{\mathbf{#1}}

\newcommand\vz{\bm{0}}
\newcommand\va{\bm{a}}
\newcommand\vb{\bm{b}}
\newcommand\vc{\bm{c}}
\newcommand\vl{\bm{\lambda}}

\newcommand\cL{\mathcal{L}}

\newcommand\Span{\operatorname{span}}
\newcommand\s{\Span}

\title{Combinatorics of generalized orthogonal polynomials of type \( R_{II} \)}
\author{Jang Soo Kim}
\address{Department of Mathematics, Sungkyunkwan University, Suwon,
  South Korea}
\email{jangsookim@skku.edu}
\author{Minho Song}
\address{Department of Mathematics, Sungkyunkwan University, Suwon,
  South Korea}
\email{smh3227@skku.edu}

\begin{document}

\maketitle

 \begin{abstract}
   In 1995, Ismail and Masson introduced orthogonal polynomials of
   types \( R_I \) and \( R_{II} \), which are defined by specific
   three-term recurrence relations with additional conditions. Recently,
   Kim and Stanton found a combinatorial interpretation for the
   moments of orthogonal polynomials of type \( R_I \) in the spirit
   of the combinatorial theory of orthogonal polynomials due to
   Flajolet and Viennot. In this paper, we push this combinatorial
   model further to orthogonal polynomials of type \( R_{II} \).
   Moreover, we generalize orthogonal polynomials of type \( R_{II} \)
   by relaxing some of their conditions. Under a basis
   condition, we then prove a master theorem that generalizes
   combinatorial models for generalized moments of various types of
   orthogonal polynomials; ordinary orthogonal polynomials, Laurent
   biorthogonal polynomials, and orthogonal polynomials of types
   \( R_I \) and \( R_{II} \). We also give a combinatorial
   interpretation of the generalized dual coefficients.
 \end{abstract}



\section{Introduction}\label{sec:intro}

An \emph{ordinary orthogonal polynomial sequence}
\( (P_n(x))_{n\ge0} \) is a sequence of polynomials such that
\( \deg P_n(x)=n \) and \( \cL(P_n(x)P_m(x)) = \kappa_n\delta_{n,m} \)
with \( \kappa_n\ne0 \), for some linear functional \( \cL \) defined
on the space of polynomials. It is well known that a monic polynomial
sequence is orthogonal if and only if it satisfies a three-term
recurrence relation
\begin{equation}\label{eq:OPS_3term}
  P_{n+1}(x) = (x - b_n) P_n(x) - \lambda_n P_{n-1}(x), \quad n\ge0,
\end{equation}
where \( P_{-1}(x) = 0, P_0(x) = 1 \) and \( \lambda_n \ne 0 \) for
\( n\ge1 \). Well-known examples of orthogonal polynomials include the
Hermite, Laguerre, and Jacobi polynomials. These are usually called
\emph{classical orthogonal polynomials}. For more information, see
\cite{IsmailBook}.

In this paper, we focus on two quantities of orthogonal polynomials,
namely moments and dual coefficients.
The values \(\mu_n=\cL(x^n)\) are
called the \emph{moments} of the orthogonal polynomials.
Here, we take \(\cL(1)=1\). By the
combinatorial theory of orthogonal polynomials due to Flajolet
\cite{Flajolet1980} and Viennot \cite{ViennotLN}, the moment \(\mu_n\)
is a generating function for the Motzkin paths from \((0,0)\) to
\((n,0)\).

For \(m\ge0\), let
\begin{equation}\label{eq:dm}
  d_m=\lambda_1\cdots\lambda_m=\cL(P_m(x)^2),
  \qquad Q_m(x)=\frac{P_m(x)}{d_m}.
\end{equation}
The \emph{generalized moments} are defined by
\begin{equation}\label{eq:gen_mom}
  \mu_{n,r,s}=\cL(x^nP_r(x)Q_s(x))
  =\frac{\cL(x^nP_r(x)P_s(x))}{\cL(P_s(x)^2)}.
\end{equation}
Viennot \cite{ViennotLN} showed that \(\mu_{n,r,s}\) is a generating
function for the Motzkin paths from \((0,r)\) to \((n,s)\).

For any polynomial sequence \((P_n(x))_{n\ge0}\) with
\(\deg P_n(x)=n\), its
coefficients \( \rho_{n,k} \) and \emph{dual coefficients}
\( \tau_{n,k} \) are defined by
\[
  P_n(x)=\sum_{k\ge0}\rho_{n,k}x^k,
  \qquad
  x^n=\sum_{k\ge0}\tau_{n,k}P_k(x).
\]
The two lower triangular matrices \((\rho_{n,k})\) and
\((\tau_{n,k})\) are mutual inverses, which explains the term dual
coefficients. The \emph{generalized dual coefficients}
\(\tau_{n,r,s}\) are defined by
\begin{equation}\label{eq:23}
  x^nP_r(x)=\sum_{s\ge0}\tau_{n,r,s}P_s(x).
\end{equation}

For an ordinary orthogonal polynomial sequence, orthogonality gives
\(\cL(P_kQ_s)=\delta_{k,s}\), and hence \( \mu_{n,r,s}=\tau_{n,r,s} \). However, generalized moments and generalized dual
coefficients need not coincide for orthogonal polynomials of types
\(R_I\) and \(R_{II}\).
In this paper, we give combinatorial interpretations of generalized
moments and generalized dual coefficients for orthogonal polynomials
of type \(R_{II}\) and their generalizations. We first recall the
definitions and known results for types \(R_I\) and \(R_{II}\).

\medskip

In 1995, Ismail and Masson~\cite{IsmailMasson} introduced orthogonal
polynomials of types \( R_I \) and \( R_{II} \). \emph{Orthogonal
  polynomials \( P^{I}_n(x) \) of type \( R_I \)} are defined by the
three-term recurrence
  \begin{equation}\label{eq:3rr-1}
    P^{I}_{n+1}(x) = (x - b_n) P^{I}_n(x) - (a_nx + \lambda_n) P^{I}_{n-1}(x), \quad n\ge0, 
\end{equation}
where \( P^{I}_{-1}(x) = 0 \) and \( P^{I}_0(x) = 1 \), and
\( a_n\ne0 \) and \( P^I_n(-\lambda_n/a_n)\ne0 \) for \( n\ge1 \). Let
\[
    d^{I}_m(x)=\prod_{i=1}^m(a_ix+\lambda_i),\qquad  Q^{I}_m(x)=\frac{P^{I}_m(x)}{d^{I}_m(x)}.
\]
Note that \( d^{I}_m(x) \) is the product of the recurrence
coefficients \( (a_ix+\lambda_i) \) in \eqref{eq:3rr-1}, so the
rational function \( Q^{I}_m(x) \) is a natural generalization of the
polynomial \( Q_m(x) \) in \eqref{eq:dm}.

\begingroup\emergencystretch=1em
Ismail and Masson~\cite[Theorem~2.1]{IsmailMasson} (see also
\cite[Theorem~2.1]{kimstanton:R1}) showed that there is a unique
linear functional \( \cL^{I} \) on the vector space of rational functions
\( W^{I} = \Span\{x^nQ^{I}_m(x): n,m\ge0\} \) such that \( \cL^{I}(1) = 1 \)
and \( \cL^{I}(x^n Q^{I}_m(x)) = 0 \) for \( 0\le n<m \).
\par\endgroup

Kim and Stanton \cite{kimstanton:R1} presented a combinatorial theory
of orthogonal polynomials of type \( R_{I} \). In particular, they
found combinatorial interpretations of the moments
\( \mu^{I}_n=\cL^{I}(x^n) \) and the generalized moments
\( \mu^{I}_{n,r,s}=\cL^{I}(x^n P^{I}_r(x)Q^{I}_s(x)) \) using
Motzkin--Schr\"oder paths. This generalizes a result of Kamioka
\cite{Kamioka2007} that the moments of Laurent biorthogonal
polynomials are generating functions for Schr\"oder paths.

\medskip

\emph{Orthogonal polynomials \( P^{II}_n(x) \) of type
  \( R_{II} \)} are defined by
\( P^{II}_{-1}(x) = 0, P^{II}_0(x) = 1 \), and
  \begin{equation}\label{eq:3rr-2}
    P^{II}_{n+1}(x) = (x - b_n) P^{II}_n(x) - (c_nx^2 + a_nx + \lambda_n) P^{II}_{n-1}(x), \quad n\ge0,
\end{equation}
with the assumptions that \( c_n\ne0 \), \( P^{II}_n(\alpha_n)\ne0 \),
and \( P^{II}_n(\beta_n)\ne0 \) for \( n\ge1 \), where \( \alpha_n \)
and \( \beta_n \) are the zeros of \( c_nx^2+a_nx+\lambda_n \). Let
\[
    d^{II}_m(x)=\prod_{i=1}^m(c_ix^2+a_ix+\lambda_i),\qquad  Q^{II}_m(x)=\frac{P^{II}_m(x)}{d^{II}_m(x)}.
\]
Ismail and Masson~\cite[Theorem~3.5]{IsmailMasson} showed that,
for fixed \( N_0 \) and \( N_1 \), there
is a unique linear functional \( \cL^{II} \) on the vector space of
rational functions
\begin{equation}\label{eq:VII}
  V^{II}=\Span\{x^nQ^{II}_m(x):0\le n\le m\}
\end{equation}
such that \( \cL^{II}(1) = N_0 \), \( \cL^{II}(xQ^{II}_1(x)) = N_1 \),
and \( \cL^{II}(x^n Q^{II}_m(x)) = 0 \) for \( 0\le n<m \). See
\cite{Ismail2019,Shukla2023} for recent work on specific orthogonal
polynomials of type \( R_{II} \).

In contrast to the type \(R_I\) case, no
analogous combinatorial theory has been developed for orthogonal
polynomials of type \(R_{II}\). To give a unified theory, we consider the larger
class of \emph{generalized orthogonal polynomials of type \(R_{II}\)}.
These are polynomials \(P_n(x)\) satisfying the same recurrence and
initial conditions as in \eqref{eq:3rr-2}, with only the condition that
\(c_nx^2+a_nx+\lambda_n\) is not the zero polynomial for \(n\ge1\).
This class contains ordinary orthogonal polynomials, Laurent
biorthogonal polynomials, and orthogonal polynomials of types
\(R_I\) and \(R_{II}\).

There are two differences from the ordinary and type \(R_I\) cases.
First, \(P_n(x)\) need not be monic; for example, the coefficient of
\(x^2\) in \(P_2(x)\) is \(1-c_1\). Second, the orthogonality
conditions alone need not determine the generalized moments.
For example, in the original type \(R_{II}\) case, the orthogonality
conditions are stated with respect to a linear functional \(\cL\) on
\(V=\Span\{x^nQ_m(x):0\le n\le m\}\), which
does not contain \(x^n\) for \(n\ge1\), so
\(\mu_n=\cL(x^n)\) is not even defined. To define the generalized
moments \(\mu_{n,r,s}=\cL(x^nP_r(x)Q_s(x))\),
we consider a linear functional on the larger space
\begin{equation}\label{eq:WII}
  W=\Span\{x^nQ_m(x): n,m\ge0\}.
\end{equation}

Our master theorem, \Cref{thm:master}, expresses the generalized
moments \(\mu_{n,r,s}\) for the four classes above as sums of the
weights of \(R_{II}\) paths from \((0,r)\) to \((n,s)\), with the
corresponding specializations of the weights. For orthogonal
polynomials of type \(R_{II}\), the restriction of \(\cL\) to \(V\)
satisfies the orthogonality conditions above, with \( \cL(1) = N_0 \)
and \( \cL(xQ_1(x)) = N_1 \) given by the sums of the weights of
\(R_{II}\) paths from \((0,0)\) to \((0,0)\) and \((1,1)\),
respectively; see \Cref{sec:comb-interpr-gener}.

\medskip

Since the generalized dual coefficients coincide with the generalized
moments for ordinary orthogonal polynomials, the Motzkin path model
shows that the generalized dual coefficients
are polynomials with nonnegative integer coefficients in
\((b_0,b_1,\ldots)\) and \((\lambda_1,\lambda_2,\ldots)\).
See \cite{LHT,Corteel2023} for recent work on generalized dual coefficients.

For generalized orthogonal polynomials \(P_n(x)\) of type
\(R_{II}\), however, the generalized dual coefficients
\(\tau_{n,r,s}\) need not be polynomials. We regard
\(a_n,b_n,c_n,\lambda_n\) as indeterminates and expand
\(\tau_{n,r,s}\) as a formal power series. For example,
if \(b_n=1\), \(a_n=\lambda_n=0\), and \(c_n=c\), where \(c\) is an
indeterminate, then
\[
  \tau_{4,0,3}
  =\frac{4-6c}{(1-2c)(1-3c+c^2)}
  =4+14c+42c^2+120c^3+\cdots.
\]
The nonnegativity of the coefficients is not clear from this expression.
Combinatorially,
we show that the generalized dual coefficient \(\tau_{n,r,s}\) is
the sum of the weights of restricted \(R_{II}\) paths from
\((0,r)\) to \((n,s)\). It follows that
\(\tau_{n,r,s}\) is a formal power series in \(a_i,b_i,c_i,\lambda_i\)
with nonnegative integer coefficients.
The paths and their weights are defined in \Cref{sec:Pre}.

\medskip

The structure of the paper is as follows. In~\Cref{sec:Pre}, we
provide basic definitions and review some known results.
In~\Cref{sec:comb-interpr-gener}, we construct a linear functional on
\(W\) for orthogonal polynomials of type \(R_{II}\). We then give a
combinatorial interpretation for the generalized moments
\(\mu_{n,r,s}\) in terms of \(R_{II}\) paths.
In \Cref{sec:gener-orth-polyn}, we generalize this construction and prove a master
theorem that includes the four classes of orthogonal polynomials discussed above.
In~\Cref{subsec:Comb}, we give a combinatorial interpretation for the
generalized dual coefficients of generalized orthogonal
polynomials of type \(R_{II}\) in terms of restricted \(R_{II}\) paths.
In~\Cref{subsec:conv}, we give sufficient conditions for the absolute
convergence of the sums of weights of \(R_{II}\) paths and restricted
\(R_{II}\) paths.
Finally, in~\Cref{sec:moments-with-const}, we show that, when the
recurrence coefficients are constant, the type \(R_{II}\) moments
coincide, up to a common factor, with moments of certain ordinary
orthogonal polynomials and of certain orthogonal polynomials of type
\(R_I\).

\section{Preliminaries}\label{sec:Pre}

In this section, we introduce basic definitions and known results. We
start with the definition of orthogonal polynomials of type
\( R_{II} \), introduced by Ismail and Masson~\cite{IsmailMasson}.

\begin{definition}\label{def:R2}
  We say that \( (P_n(x))_{n\ge0} \) is a sequence of \emph{orthogonal polynomials of
    type \( R_{II} \)} if
  \( P_{-1}(x) = 0, P_0(x) = 1 \), and
\begin{equation}\label{eq:three_term}
   P_{n+1}(x) = (x - b_n) P_n(x) - (c_nx^2 + a_nx + \lambda_n) P_{n-1}(x), \quad n\ge0, 
\end{equation}
for some sequences \( \vb=(b_0,b_1,\ldots) \),
\( \vl=(\lambda_1,\lambda_2,\ldots)\), \( \va=(a_1,a_2,\ldots) \), and
\( \vc=(c_1,c_2,\ldots) \) such that \( c_n\neq0 \),
\( P_n(\alpha_n)\ne0 \), and \( P_n(\beta_n)\ne0 \) for \( n\ge1 \),
where \( \alpha_n \) and \( \beta_n \) are the roots of
\( c_nx^2 + a_nx + \lambda_n \).

Given the polynomials \( P_n(x) \), we will use the following
notation:
\begin{align*}
  d_n(x)&=\prod_{i=1}^n(c_ix^2+a_ix+\lambda_i), & Q_n(x)&=\frac{P_n(x)}{d_n(x)},\\
  e_n(x)&=\prod_{i=1}^n(x-\alpha_i), & f_n(x)&=\prod_{i=1}^n(x-\beta_i).
\end{align*}
Here, we use the standard convention that
the empty product is defined to be \( 1 \); for example,
\( d_0(x)=1 \).
We also define the vector space
\[
  V=\s\{x^nQ_m(x):0\le n\le m\},
\]
and its extension
\begin{equation}\label{eq:W}
  W =\s \{x^nQ_m(x):n,m\ge0\}.
\end{equation}
\end{definition}

Ismail and Masson \cite{IsmailMasson} showed that orthogonal
polynomials of type \( R_{II} \) have partial orthogonality as
follows.

\begin{theorem} \cite[Theorem 3.5]{IsmailMasson} For the polynomials
  \( P_n(x) \) defined in \Cref{def:R2} and for any numbers \( N_0 \)
  and \( N_1 \), there is a unique linear functional \(\LL\) on
  \( V \) such that \( \LL(1)=N_0 \), \( \LL(xQ_1(x))=N_1 \), and
  \( \LL(x^n Q_m(x)) = 0 \) for \( 0\le n<m \).
\end{theorem}

Now we review known combinatorial properties of these orthogonal
polynomials, which will be generalized to orthogonal polynomials of
type \( R_{II} \) in this paper. To put them all in one framework, we
introduce \( R_{II} \) paths.

A \emph{lattice path} is a sequence \( p=(p_0,p_1,\dots,p_n) \) of
points \( p_i\in \ZZ \times \ZZ_{\ge0} \). For \( i\ge1 \), the
\( i \)th \emph{step} \( s_i \) of \( p \) is defined by
\( s_i = p_{i} - p_{i-1} \). We will also consider the path \( p \) as
the sequence \( s_1\dots s_n \) of its steps (together with its
starting point \( p_0 \)).

\begin{definition}\label{def:gen_MS}
  An \emph{\( R_{II} \) path} is a lattice path consisting of \emph{up
    steps} \( U=(1,1) \), \emph{horizontal steps} \( H=(1,0) \),
  \emph{down steps} \( D=(1,-1) \), \emph{vertical-down steps}
  \( V=(0,-1) \), and \emph{backward-down steps} \( B=(-1,-1) \). Denote by
  \( \gm_{n,r,s} \) the set of \( R_{II} \) paths from \( (0,r) \) to
  \( (n,s) \). We also write \( \gm_{k,m}=\gm_{k,0,m} \) and
  \( \gm_k = \gm_{k,0,0} \).
\end{definition}

For later use, we also define restricted \( R_{II} \) paths.

\begin{definition}\label{def:rgm}
  A \emph{restricted \( R_{II} \) path} is an \( R_{II} \)-path in
  \( \gm_{n,r,s} \) that immediately ends when it touches the line
  \( x=n \).  We denote by \( \rgm_{n,r,s} \) the set of restricted
  \( R_{II} \) paths in \( \gm_{n,r,s} \). We call the condition that
  the path ends when it first touches the line \( x=n \) the
  \emph{restricted condition}.
\end{definition}

For an \( R_{II} \) path \( p = s_1\dots s_k \), we define the
\emph{weight} \( \wt(p) \) of \( p \) by
\( \wt(p)=\prod_{t=1}^k\wt(s_t) \), where
\begin{equation}\label{eq:weight_GM}
  \wt(s_t)=\begin{cases}
            1 & \mbox{if } s_t=U,\\
            b_i & \mbox{if } s_t=H \mbox{ starting at height \( i \)},\\
            \lambda_i & \mbox{if } s_t=D \mbox{ starting at height \( i \)},\\
            a_i & \mbox{if } s_t=V \mbox{ starting at height \( i \)},\\
            c_i & \mbox{if } s_t=B \mbox{ starting at height \( i \)}.
           \end{cases}
\end{equation}
See \Cref{fig:generalized_Motz_Sch}. Note that steps \( U \) and
\( B \) can be overlapped in an \( R_{II} \) path as shown in
\Cref{fig:generalized_Motz_Sch}. Thus, there can be infinitely many
elements in \( \gm_{n,r,s} \) and
\( \sum_{p\in \gm_{n,r,s}}\wt(p) \) is a formal power series in the
variables in \( \vb,\vl,\va \), and \( \vc \).

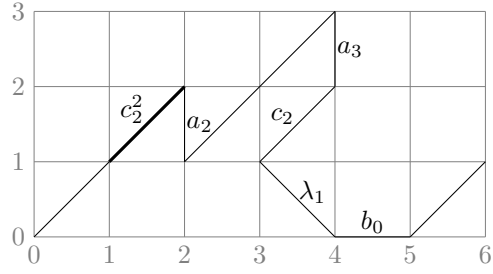
\begin{figure}
  \tikzset{
    set arrow inside/.code={\pgfqkeys{/tikz/arrow inside}{#1}},
    set arrow inside={end/.initial=>, opt/.initial=},
    /pgf/decoration/Mark/.style={
        mark/.expanded=at position #1 with
        {
          \noexpand\arrow[\pgfkeysvalueof{/tikz/arrow inside/opt}]{\pgfkeysvalueof{/tikz/arrow inside/end}}
        }
    },
    arrow inside/.style 2 args={
        set arrow inside={#1},
        postaction={
            decorate,decoration={
                markings,Mark/.list={#2}
            }
        }
    },
}
  \centering
  \begin{tikzpicture}[scale=1]
   \foreach \i in {\xMin,...,\xMax} {
        \draw [very thin,gray] (\i,\yMin) -- (\i,\yMax)  node [below] at (\i,\yMin) {$\i$};
    }
    \foreach \i in {\yMin,...,\yMax} {
        \draw [very thin,gray] (\xMin,\i) -- (\xMax,\i) node [left] at (\xMin,\i) {$\i$};
    }

    \draw (0,0) -- (1,1);
    \draw[very thick] (1,1) -- (2,2);
    \draw (2,2) -- (2,1);
    \draw (2,1) -- (4,3);
    \draw (4,2) -- (4,3);
    \draw (3,1) -- (4,2);
    \draw (3,1) -- (4,0);
    \draw (4,0) -- (5,0);
    \draw (5,0) -- (6,1);

    \node at (1.3,1.7){\( c_2^2 \)}; 
    \node at (2.2,1.5){\( a_2 \)};
    \node at (4.2,2.5){\( a_3 \)};
    \node at (3.3,1.6){\( c_2 \)};
    \node at (3.7,.6){\( \lambda_1 \)};
    \node at (4.5,.2){\( b_0 \)};
  \end{tikzpicture}
  \caption{An \( R_{II} \) path \( P = UUBUBUVUUVBDHU\in\gm_{6,1} \)
    with \( \wt(P)=a_2a_3b_0c^3_2\lambda_1. \) The thick segment from
    \( (1,1) \) to \( (2,2) \) represents \( UBUBU \).}
  \label{fig:generalized_Motz_Sch}
\end{figure}

A \emph{Motzkin path} is an \( R_{II} \) path without vertical-down
steps and backward-down steps. A \emph{Schr\"oder path} is an \( R_{II} \)
path without down steps and backward-down steps. A
\emph{Motzkin--Schr\"oder path} is an \( R_{II} \) path without
backward-down steps. Let \( \Mot_{n,r,s} \) (resp.~\( \Sch_{n,r,s} \) and
\( \MS_{n,r,s} \)) denote the set of Motzkin paths (resp. Schr\"oder
paths and Motzkin--Schr\"oder paths) from \( (0,r) \) to \( (n,s) \).

We are now ready to state known results on various orthogonal
polynomials. Note that for the case of ordinary orthogonal
polynomials, by the well-known fact
\( \LL(P_n(x)^2) = \lambda_1 \cdots \lambda_n \), we have
\[
  Q_n(x) = \frac{P_n(x)}{d_n(x)}
  = \frac{P_n(x)}{\lambda_1 \cdots \lambda_n}
  = \frac{P_n(x)}{\LL(P_n(x)^2)}.
\]

\begin{theorem}\label{thm:mot} \cite[Proposition 17, pp. I–15]{ViennotLN}
  Let \( P_n(x) \) be the ordinary orthogonal polynomials
  given by \eqref{eq:three_term}
  with \( \vc = \va = \vz \). Then there
  is a linear functional \( \LL \) on \( W \) such that
  \[
    \LL(x^n P_r(x) Q_s(x)) =
    \sum_{p\in \Mot_{n,r,s}}  \wt(p).
  \]
\end{theorem}

\begin{theorem}\label{thm:sch} \cite[Theorem~17]{Kamioka2007}
  Let \( P_n(x) \) be the Laurent biorthogonal polynomials given by
  \eqref{eq:three_term} with \( \vc = \vl = \vz \). Then there is a
  linear functional \( \LL \) on \( W \) such that
  \[
    \LL(x^n P_r(x) Q_s(x)) =
    \sum_{p\in \Sch_{n,r,s}}  \wt(p).
  \]
\end{theorem}

We note that \Cref{thm:sch} is an equivalent statement of
\cite[Theorem~17]{Kamioka2007}; see \cite[Theorem~4.7]{kimstanton:R1}.

\begin{theorem}\label{thm:MS} \cite[Theorem~3.12]{kimstanton:R1}
  Let \( P_n(x) \) be the orthogonal polynomials of type \( R_I \) given
  by \eqref{eq:three_term} with \( \vc = \vz \). Then there is a
  linear functional \( \LL \) on \( W \) such that
  \[
    \LL(x^n P_r(x) Q_s(x)) =
    \sum_{p\in \MS_{n,r,s}}  \wt(p).
  \]
\end{theorem}

Finally, we introduce dual coefficients of polynomials.

\begin{definition}
  Let \( (P_n(x))_{n\ge0} \) be a sequence of polynomials such that
  \( \deg P_n(x) = n \) for all \( n\ge0 \). The \emph{coefficients}
  \( \rho_{n,k} \) and \emph{dual coefficients} \( \tau_{n,k} \) of
  the polynomials \( P_n(x) \) are defined by
\begin{align}\label{eq:mixed_moments}
  P_n(x)=\sum_{k\ge0}\rho_{n,k}x^k, \quad  x^n=\sum_{k\ge0}\tau_{n,k}P_k(x).
\end{align}
We also define
the \emph{(generalized) dual coefficients}
\( \tau_{n,r,s} \) by
\begin{equation}\label{eq:6}
  x^nP_r(x)=\sum_{s\ge0}\tau_{n,r,s}P_s(x).
\end{equation}
\end{definition}

Let \( P_n(x) \) be the ordinary orthogonal polynomials
given by \eqref{eq:three_term} with \( \vc = \va = \vz \), and let
\( \LL \) be the corresponding linear functional. The orthogonality
relation is
\[
  \LL(P_r(x)Q_s(x))=\delta_{r,s}.
\]
Thus, multiplying both sides of \eqref{eq:6} by \( Q_s(x) \) and
applying \( \LL \) yields
\begin{equation}\label{eq:7}
  \LL(x^nP_r(x)Q_s(x))= \tau_{n,r,s}.
\end{equation}
Thus, \Cref{thm:mot} gives the following corollary.

\begin{corollary}\label{cor:mot2}
  Let \( P_n(x) \) be the ordinary orthogonal polynomials
  given by \eqref{eq:three_term}
  with \( \vc = \va = \vz \). Then 
  \[
    \tau_{n,r,s} = \sum_{p\in \Mot_{n,r,s}}  \wt(p).
  \]
\end{corollary}

We note that \eqref{eq:7} no longer holds for Laurent biorthogonal
polynomials or orthogonal polynomials of type \( R_I \) or
\( R_{II} \). In \Cref{subsec:Comb}, we will find an analogous result of
\Cref{cor:mot2} for orthogonal polynomials of type \( R_{II} \) using
restricted \( R_{II} \) paths.

\section{A combinatorial interpretation for generalized moments}
\label{sec:comb-interpr-gener}

Throughout this section we follow the notation in \Cref{def:R2}.
Recall the vector space
\[
  W =\s \{x^nQ_m(x):n,m\ge0\}.
\]

\subsection{A basis for \texorpdfstring{\(W\)}{W}}
\label{subsec:extended_V}

In this subsection, we provide a basis for \( W \) as follows.

\begin{theorem}\label{thm:basis_V}
  The vector space \( W = \s \{x^nQ_m(x):n,m\ge0\} \) has a basis
  \[
    {\mathcal B} = \{x^n:n\ge0\}\cup\{Q_n(x):n\ge1\}\cup\{xQ_n(x):n\ge1\}.
  \]
\end{theorem}

To prove this theorem, we consider
\begin{align*}
 W'=&\s \{x^n/d_m(x):n,m\ge0\}.
\end{align*}
Note that since \( Q_m=P_m/d_m \) and \( P_m \) is a
polynomial, every \( x^nQ_m \) belongs to \( W' \) (in particular, so
do \( Q_m \) and \( xQ_m \)). Hence \( W \) is contained in \( W' \).

The proof of \Cref{thm:basis_V} proceeds in three steps. First, we find
a basis \(\mathcal B'\) of \(W'\). Second, we prove an equality of
finite-dimensional spans that implies \(W=W'\).
Finally, we use these results to prove that \(\mathcal B\) spans \(W\)
and is linearly independent.

\begin{lemma}\label{lem:V' basis}
 The vector space \( W' = \s\{x^n/d_m(x):n,m\ge0\} \) has a basis
 \[
   {\mathcal B}'=\{x^n:n\ge0\}\cup\{1/d_m(x):m\ge1\}\cup\{x/d_m(x):m\ge1\}.
 \]
\end{lemma}
\begin{proof}
  Since \( 1/d_m(x) \sim x^{-2m} \) and \( x/d_m(x) \sim x^{-2m+1} \)
  asymptotically up to constants, \( {\mathcal B}' \) is linearly
  independent. Since \( {\mathcal B}'\subset W' \), it suffices to
  show that \( \s({\mathcal B}')=W' \). To see this, we claim that
  \( x^n/d_m(x)\in\s({\mathcal B}') \) for any \( n,m\ge0 \). First, we write the
  fraction as
\begin{align}\label{eq:fraction V'}
    \frac{x^n}{d_m(x)}=q(x)+\frac{r(x)}{d_m(x)}, 
\end{align}
 where \( q(x) \) and \( r(x) \) are polynomials and \( \deg(r(x))<2m \). Now it suffices to show that \( r(x)/d_m(x)\in\s({\mathcal B}') \). If \( \deg r(x) \le 1 \), then we have \( r(x)/d_m(x)\in\s({\mathcal B}') \). Otherwise, we can write
 \begin{align}\label{eq:iteration V'}   \frac{r(x)}{d_m(x)}=\frac{(c_mx^2+a_mx+\lambda_m)r_1(x)+s_1(x)}{d_m(x)}=\frac{r_1(x)}{d_{m-1}(x)}+\frac{s_1(x)}{d_m(x)},
 \end{align}
 where \( r_1(x) \) is a polynomial with \( \deg r_1(x)<2m-2 \) and \( s_1(x) \) is a polynomial with \( \deg s_1(x) \le 1 \). By iterating \eqref{eq:iteration V'}, we can express \( r(x)/d_m(x) \) as a linear combination of the elements in \( {\mathcal B}' \), which completes the proof.
\end{proof}

\begin{lemma}\label{lem:finite_spans}
  For \(M\ge0\), let
  \begin{align*}
    \mathcal B_M&=\{1\}\cup\{Q_m(x),xQ_m(x):1\le m\le M\},\\
    \mathcal B'_M&=\{1\}\cup\{1/d_m(x),x/d_m(x):1\le m\le M\}.
  \end{align*}
  Then
  \[
    \Span(\mathcal B_M)=\Span(\mathcal B'_M).
  \]
  In particular, \(W=W'\).
\end{lemma}
\begin{proof}
  For \(m\ge1\), we have \(\deg P_m\le m\) and \(\deg d_m=2m\).
  The polynomial division argument in the proof of
  \Cref{lem:V' basis} therefore expresses \(Q_m(x)\)
  and \(xQ_m(x)\) as linear combinations of \(1\) and
  \(1/d_i(x),x/d_i(x)\) with \(1\le i\le m\). Thus
  \(\Span(\mathcal B_M)\subseteq\Span(\mathcal B'_M)\).
  We prove the reverse inclusion by induction on \(M\).
  The case \(M=0\) is clear since \(\mathcal B_0=\mathcal B'_0=\{1\}\).

  Now, let \(M=m\ge1\) and assume the equality holds for \(M=m-1\).
  It suffices to show that
  \(1/d_m(x),x/d_m(x)\in\Span(\mathcal B_m)\).
  First, we claim that \(1/d_m(x)\in\Span(\mathcal B_m)\). If
  \( r_m(x) \) is the quotient of \( P_m(x) \) when divided by
  \( x-\beta_m \), then we have
 \begin{equation}\label{eq:division_beta}
  \frac{P_{m}(x)}{d_{m-1}(x)(x-\beta_m)} =\frac{r_m(x)}{d_{m-1}(x)}+\frac{P_m(\beta_m)}{d_{m-1}(x)(x-\beta_m)}.
 \end{equation}
 Since \(\deg r_m\le m-1\), the same polynomial division argument gives
 \(r_m(x)/d_{m-1}(x)\in\Span(\mathcal B'_{m-1})
 =\Span(\mathcal B_{m-1})\) by the inductive hypothesis.
 Rewriting \eqref{eq:division_beta} for the last term on its right-hand side, substituting
 \(c_m(x-\alpha_m)Q_m(x)=P_m(x)/(d_{m-1}(x)(x-\beta_m))\),
 and dividing by the nonzero number \(P_m(\beta_m)\), we obtain
 \begin{align}\label{eq:basis_ind1}
   \frac{1}{d_{m-1}(x)(x-\beta_m)}=\frac{1}{P_m(\beta_m)}\left( c_mxQ_m(x)-c_m\alpha_mQ_m(x)-\frac{r_m(x)}{d_{m-1}(x)} \right).
 \end{align}
 Thus \(1/(d_{m-1}(x)(x-\beta_m))\in\Span(\mathcal B_m)\).
 We also have
 \(1/(d_{m-1}(x)(x-\alpha_m))\in\Span(\mathcal B_m)\) in the same way.

 If \( \alpha_m\neq\beta_m, \) then
 \[
   \frac{1}{d_m(x)}=\frac{1}{c_m(\beta_m-\alpha_m)}\left( \frac{1}{d_{m-1}(x)(x-\beta_m)} - \frac{1}{d_{m-1}(x)(x-\alpha_m)} \right).
 \]
 Hence \(1/d_m(x)\in\Span(\mathcal B_m)\).
 If \( \alpha_m=\beta_m, \) then we can write \( Q_m(x) \) as
 \begin{align*}
  Q_m(x) = \frac{P_m(x)}{d_m(x)}&=\frac{c_m(x-\alpha_m)^2R_{m-2}(x)+\gamma c_m(x-\alpha_m)+P_m(\alpha_m)}{d_m(x)}\\
                  &=\frac{R_{m-2}(x)}{d_{m-1}(x)}+\frac{\gamma}{d_{m-1}(x)(x-\alpha_m)}+\frac{P_m(\alpha_m)}{d_m(x)}
 \end{align*}
for a polynomial \(R_{m-2}(x)\) of degree at most \(m-2\) and a
constant \(\gamma\), with the convention that \(R_{-1}(x)=0\).
As above, \(R_{m-2}(x)/d_{m-1}(x)\in\Span(\mathcal B_{m-1})\), so
 \[
   \frac{1}{d_m(x)}=\frac{1}{P_m(\alpha_m)}\left( Q_m(x)-\frac{R_{m-2}(x)}{d_{m-1}(x)}-\frac{\gamma}{d_{m-1}(x)(x-\alpha_m)} \right).
 \]
 Thus \(1/d_m(x)\in\Span(\mathcal B_m)\) in this case as well.

 Finally, since \(1/(d_{m-1}(x)(x-\alpha_m))
 =c_m(x-\beta_m)/d_m(x)\), we obtain
  \begin{equation}\label{eq:basis_ind2}
    \frac{x}{d_m(x)}
   =\frac{\beta_m}{d_m(x)}
    +\frac{1}{c_m d_{m-1}(x)(x-\alpha_m)}
   \in\Span(\mathcal B_m).
  \end{equation}
  This completes the induction. Since every element of \(\mathcal B'\)
 belongs to \(W\), \Cref{lem:V' basis} gives \(W'\subseteq W\).
 Together with \(W\subseteq W'\), this proves \(W=W'\).
\end{proof}

Now we prove~\Cref{thm:basis_V}. 

\begin{proof}[Proof of~\Cref{thm:basis_V}]
  By \Cref{lem:finite_spans}, \(\mathcal B'\subseteq\Span(\mathcal B)\).
  Since \(\mathcal B\subseteq W=W'\), \Cref{lem:V' basis} shows that
  \(\mathcal B\) spans \(W\).

  For linear independence, let \(M,N\ge0\). By \Cref{lem:finite_spans},
  the sets
  \[
    \mathcal B_M\cup\{x^n:1\le n\le N\}
    \quad\text{and}\quad
    \mathcal B'_M\cup\{x^n:1\le n\le N\}
  \]
  have the same span. By \Cref{lem:V' basis}, the second set is linearly
  independent and has \(2M+N+1\) elements. Since the first set has at
  most \(2M+N+1\) elements, it is also linearly independent. Every
  finite subset of \(\mathcal B\) is contained in such a set, so
  \(\mathcal B\) is linearly independent.
\end{proof}

\begin{remark}\label{rem:3}
  In \cite[Theorem 3.5]{IsmailMasson}, Ismail and Masson implicitly
  use the fact (without providing a proof) that
  \(\mathcal B\setminus\{x^n:n\ge1\}\) is a basis for \(V\).
  This follows from \Cref{lem:V' basis,lem:finite_spans}
  and \Cref{thm:basis_V}.
  Consequently,
  \(W=V\oplus\Span\{x^n:n\ge1\}\).
\end{remark}
  
\subsection{Moments and generalized moments for orthogonal polynomials of type \( R_{II} \)}
\label{subsec:mu}

In this subsection, we construct a linear functional \( \cL \) on
\( W =\s \{x^nQ_m(x):n,m\ge0\} \) that yields the partial orthogonality
relations for orthogonal polynomials of type \( R_{II} \). Moreover,
we provide a combinatorial description of the generalized moments
\( \LL(x^n P_r(x) Q_s(x)) \). The linear functional on \(V\)
considered by Ismail and Masson~\cite{IsmailMasson} does not have a
unique extension to \(W\), since its values on \(x^n\), \(n\ge1\),
must be specified separately. We choose these values combinatorially.

By~\Cref{thm:basis_V}, we can define a linear functional on \( W \) as
follows.

\begin{definition}\label{def:ext_linear_fct}
  The linear functional \( \cL \) on the vector space
  \( \s\{x^nQ_m(x):n,m\ge0\} \) is defined by
\begin{align*}
  \cL(x^n)&=\sum_{P\in\gm_n}\wt(P), \quad n\ge0,\\
  \cL(Q_n(x)) &= \cL(xQ_{n+1}(x))=0, \quad n\ge1,\\
  \cL(xQ_1(x))&=\sum_{P\in\gm_{1,1}}\wt(P).
\end{align*}
\end{definition}

\begin{definition}\label{def:gen_moments}
  The \emph{generalized moment} \( \mu_{n,r,s} \) of orthogonal
  polynomials of type \( R_{II} \) is given by
\[
  \mu_{n,r,s}=\cL(x^nP_r(x)Q_s(x)).
\]
We also write \( \mu_{n,m} = \mu_{n,0,m} \) and
\( \mu_{n} = \mu_{n,0,0} \).
\end{definition}

By multiplying both sides of \eqref{eq:three_term} by \( 1/d_n(x) \),
we have
  \begin{equation}\label{eq:Q-rr}
  (c_{n+1}x^2+a_{n+1}x+\lambda_{n+1})Q_{n+1}(x)=(x-b_{n})Q_n(x)-Q_{n-1}(x).
\end{equation}
The next proposition provides a combinatorial interpretation for the values of \( \cL \) on the set \( \{x^nQ_m(x):n,m\ge0\} \).
Recall that \(c_i\ne0\) for all \(i\ge1\) by \Cref{def:R2}.

\begin{proposition}\label{lem:mu_nm}
  For nonnegative integers \( n \) and \( m \), we have
  \[
    \mu_{n,m}=\sum_{p\in\gm_{n,m}}\wt(p).
  \]
\end{proposition}

\begin{proof}
  Let \( \mu'_{n,m} \) be the right-hand side of the equation. We will
  prove \( \mu_{n,m} = \mu'_{n,m} \) using induction on \( (n,m) \).
  By \Cref{def:ext_linear_fct}, the statement holds for \( m=0 \),
  \( n=0 \), or \( n=1 \).

  Now let \( n\ge2 \) and \( m\ge1 \), and suppose that the statement
  holds for all \( (n',m') \ne (n,m) \) with \( n'\le n \) and
  \( m'\le m \). By \eqref{eq:Q-rr},
  \[
      (c_{m}x^2+a_{m}x+\lambda_{m})Q_{m}(x)=(x-b_{m-1})Q_{m-1}(x)-Q_{m-2}(x).
  \]
  Multiplying both sides by \( x^{n-2} \) and applying \( \LL \) yields
\begin{equation}\label{eq:mu_nm_rec}
 c_{m}\mu_{n,m}+a_{m}\mu_{n-1,m}+\lambda_{m}\mu_{n-2,m}=\mu_{n-1,m-1}-b_{m-1}\mu_{n-2,m-1}-\mu_{n-2,m-2}.
\end{equation}

On the other hand, considering the last step of \( p\in\gm_{n-1,m-1} \),
we have
\[
  \mu'_{n-1,m-1} =
  c_{m}\mu'_{n,m}+a_{m}\mu'_{n-1,m}+\lambda_{m}\mu'_{n-2,m}
  +b_{m-1}\mu'_{n-2,m-1}+\mu'_{n-2,m-2},
\]
which is the same recurrence relation as \eqref{eq:mu_nm_rec}. Thus,
applying the inductive hypothesis to every term in
\eqref{eq:mu_nm_rec} except \( c_m\mu_{n,m} \) gives
\( c_{m}\mu_{n,m} = c_{m} \mu'_{n,m} \). Since \( c_m\ne 0 \), the
statement also holds for \( (n,m) \), which completes the proof.
\end{proof}

\begin{remark}\label{rem:1}
  Using the recurrence~\eqref{eq:mu_nm_rec} along
  with~\Cref{def:ext_linear_fct}, one can see that \( \mu_{n,m}=0 \)
  for \( 0\le n < m \). This shows the orthogonality of orthogonal
  polynomials of type \( R_{II} \).
\end{remark}

We are now ready to obtain a combinatorial description of
\( \mu_{n,r,s} \).

\begin{theorem}\label{thm:mu_nrs}
  For nonnegative integers \( n,r \), and \( s \), we have
  \[
    \mu_{n,r,s}=\sum_{P\in\gm_{n,r,s}} \wt(P).
  \]
\end{theorem}
\begin{proof}
  Let \( \mu'_{n,r,s}=\sum_{P\in\gm_{n,r,s}} \wt(P) \). We will prove
  \( \mu_{n,r,s} = \mu'_{n,r,s} \) by induction on \( r \). The base
  step \( r=0 \) follows from \Cref{lem:mu_nm}. For the inductive
  step, let \( k\ge0 \) and suppose that
  \( \mu_{n,r,s} = \mu'_{n,r,s} \) holds for all \( r\le k \).
  Replacing \( n \) by \( k \) in \eqref{eq:three_term}, multiplying
  both sides of the resulting equation by \( x^{n}Q_s(x) \), and
  taking \( \LL \), we obtain
  \begin{equation}\label{eq:3}
    \mu_{n,k+1,s} = \mu_{n+1,k,s}-b_k \mu_{n,k,s} - c_k \mu_{n+2,k-1,s} - a_k \mu_{n+1,k-1,s} - \lambda_k \mu_{n,k-1,s}.
\end{equation}

On the other hand, considering the first step of
\( p\in \gm_{n+1,k,s} \), we have
\[
   \mu'_{n+1,k,s}  = \mu'_{n,k+1,s}+b_k \mu'_{n,k,s} + c_k \mu'_{n+2,k-1,s} + a_k \mu'_{n+1,k-1,s} + \lambda_k \mu'_{n,k-1,s},
\]
which is the same recurrence as \eqref{eq:3}. Thus, by the inductive
hypothesis, \( \mu_{n,r,s} = \mu'_{n,r,s} \) also holds for \( r=k+1 \).
By induction, we obtain the theorem.
\end{proof}

\section{Generalized orthogonal polynomials of type \( R_{II} \)}
\label{sec:gener-orth-polyn}

In this section, we extend the definition of orthogonal polynomials of
type \( R_{II} \) by relaxing some of their conditions.

\begin{definition}\label{def:R2-2}
  We say that \( (P_n(x))_{n\ge0} \) is a sequence of \emph{generalized
  orthogonal polynomials of type \( R_{II} \)}, or \emph{\( R_{II} \)
    polynomials} for short, if \( P_{-1}(x) = 0, P_0(x) = 1 \), and
\[
   P_{n+1}(x) = (x - b_n) P_n(x) - (c_nx^2 + a_nx + \lambda_n) P_{n-1}(x), \quad n\ge0, 
\]
where
\( \va=(a_1,a_2,\ldots), \vb=(b_0,b_1,\ldots), \vc=(c_1,c_2,\ldots)\),
and \( \vl=(\lambda_1,\lambda_2,\ldots)\) are sequences such that,
for all \( n\ge1 \), at least one of \( \lambda_n \), \( a_n \), or
\( c_n \) is nonzero.
We will use the following
notation:
\[
  d_n(x) = \prod_{i=1}^n(c_ix^2+a_ix+\lambda_i), \qquad   Q_n(x)=\frac{P_n(x)}{d_n(x)}.
\]
\end{definition}

Note that \( R_{II} \) polynomials
contain various types of orthogonal polynomials:
\begin{itemize}
\item Orthogonal polynomials of type \( R_{II} \) are \( R_{II} \)
polynomials such that 
\( c_n\neq0, P_n(\alpha_n)\ne0 \), and \( P_n(\beta_n)\ne0 \) for
\( n\ge1 \), where \( \alpha_n \) and \( \beta_n \) are the roots of
\( c_nx^2 + a_nx + \lambda_n \).
\item Orthogonal polynomials of type \( R_{I} \) are \( R_{II} \)
  polynomials such that \( c_n=0, a_n\ne 0 \), and
  \( P_n(-\lambda_n/a_n)\ne0 \) for \( n\ge1 \).
\item Laurent biorthogonal polynomials are \( R_{II} \) polynomials
  such that \( c_n=\lambda_n=0, a_n\ne 0 \), and \( P_n(0)\ne0 \)
  for \( n\ge1 \).
\item Ordinary orthogonal polynomials are \( R_{II} \) polynomials
  such that \( c_n=a_n=0 \) and \( \lambda_n\ne0 \) for \( n\ge1 \).
\end{itemize}

Let
\[
  W = \Span \{x^n Q_m(x): n,m\ge 0\}.
\]
Recall from \eqref{eq:Q-rr} that, for \( k\ge0 \),
\begin{equation}\label{eq:5}
  (c_{k+1}x^2+a_{k+1}x+\lambda_{k+1})Q_{k+1}(x)=(x-b_{k})Q_k(x)-Q_{k-1}(x),
\end{equation}
where \( Q_{-1}(x) = 0 \).
If we multiply both sides by \( x^\ell \), we obtain
\begin{multline}\label{eq:8}
  c_{k+1}x^{\ell+2}Q_{k+1}(x)
  + a_{k+1}x^{\ell+1}Q_{k+1}(x)+\lambda_{k+1}x^\ell Q_{k+1}(x)\\
  =x^{\ell+1}Q_k(x) -b_{k}x^\ell Q_k(x)- x^\ell Q_{k-1}(x).
\end{multline}

Let \(S \subseteq \ZZ_{\ge0}^2 \) and
\( B = \{x^n Q_m(x): (n,m)\in S\} \). We say that \( x^n Q_m(x) \)
\emph{is obtained from \( B \) by a single application of
  \eqref{eq:8}} if there exist \( k,\ell \ge0 \) such that
\( x^n Q_m(x) \) appears in \eqref{eq:8} with a nonzero coefficient
and, for every term \( x^{n'} Q_{m'}(x) \) in \eqref{eq:8} with
\( (n',m')\ne (n,m) \), either its coefficient is zero or it is
contained in \( B \). We say that \( x^n Q_m(x) \) \emph{is obtained
  from \( B \) by multiple applications of \eqref{eq:8}} if there is a
sequence \( (n_1,m_1),(n_2,m_2),\dots,(n_t,m_t) \) with
\( (n_t,m_t) = (n,m) \) such that \( x^{n_i} Q_{m_i}(x) \) is obtained from
  \( B_{i-1} \) by a single application of \eqref{eq:8} for all
\( i=1,2,\dots,t \), where \( B_0=B \) and
\( B_{i} = B_{i-1} \cup\{x^{n_i} Q_{m_i}(x)\} \).

\begin{definition}
  A basis \( B \) of \( W \) is said to be \emph{good} if
  \( B = \{x^n Q_m(x): (n,m)\in S\} \) for a subset
  \( S \subseteq \ZZ_{\ge0}^2 \) and, for every \( n,m\ge0 \),
  \( x^n Q_m(x) \) is obtained from \( B \) by multiple applications
  of \eqref{eq:8}. 
\end{definition}

\begin{theorem}\label{thm:master}
  Suppose that there is a good basis of \( W \). Then there is a
  linear functional \( \LL \) on \( W \) such that, for all
  \( n,r,s\ge0 \),
\[
  \LL(x^n P_r(x) Q_s(x)) = \sum_{p\in \gm_{n,r,s}}
  \wt(p).
\]
\end{theorem}

\begin{proof}
Let \( B = \{x^n Q_m(x): (n,m)\in S\} \) be a good basis of
\( W \).
Define \( \LL \) on \( B \) by
\begin{equation}\label{eq:4}
  \LL(x^n Q_m(x)) = R_{n,m},
\end{equation}
where \( R_{n,m} = \sum_{p\in \gm_{n,m}} \wt(p) \). Since \( B \) is a
basis, we can extend the definition of \( \LL \) to \( W \) by
linearity. We claim that \eqref{eq:4} holds for all \( n,m\ge0 \).

To prove the claim, first suppose that \( x^n Q_m(x) \) is obtained
from \( B \) by a single application of \eqref{eq:8}. Applying \( \LL \)
to both sides of \eqref{eq:8} gives
\begin{equation}\label{eq:9}
    c_{k+1}R'_{\ell+2,k+1}
  + a_{k+1}R'_{\ell+1,k+1} + \lambda_{k+1} R'_{\ell,k+1}
  = R'_{\ell+1,k} -b_{k} R'_{\ell,k}- R'_{\ell,k-1},
\end{equation}
where \( R'_{i,j} = \LL(x^i Q_j(x)) \).
By the construction of \( \LL \), if \( x^i Q_j(x)\in B \),
then \( R'_{i,j} = R_{i,j} \).
Considering the last step of any path \( p\in \gm_{\ell+1,k} \), we have
\begin{equation}\label{eq:10}
  R_{\ell+1,k} = 
    c_{k+1}R_{\ell+2,k+1}
  + a_{k+1}R_{\ell+1,k+1} + \lambda_{k+1} R_{\ell,k+1}
  +b_{k} R_{\ell,k} + R_{\ell,k-1}.
\end{equation}
Since every term \( R'_{i,j} \) with \( (i,j)\ne (n,m) \) appearing
in \eqref{eq:9} satisfies \( R'_{i,j} = R_{i,j} \), comparing
\eqref{eq:9} and \eqref{eq:10} gives \(  R'_{n,m} = R_{n,m} \),
hence \( \LL(x^n Q_m(x)) = R_{n,m} \).

Now consider \( x^n Q_m(x) \) for arbitrary \( n,m\ge0 \). Since \( B \) is
a good basis, \( x^n Q_m(x) \) is obtained from \( B \) by multiple
applications of \eqref{eq:8}. Applying the previous argument
iteratively, we obtain that \( \LL(x^n Q_m(x)) = R_{n,m} \).
Therefore \eqref{eq:4} holds for all \( n,m\ge0 \), as claimed.

This shows the theorem for the case \( r=0 \). The general case
\( r\ge0 \) can be proved by induction using the same arguments in the
proof of \Cref{thm:mu_nrs}.
\end{proof}

Since \( \gm_{n,0,m}=\emptyset \) for \( n<m \), we
obtain the following orthogonality.

\begin{corollary}\label{cor:master}
  Suppose that there is a good basis of \( W \). Then there is a
  linear functional \( \LL \) on \( W \) such that \( \LL(1) \) is a
  power series in \( \vc=(c_1,c_2,\dots) \) with constant term \( 1 \)
  and, for all \( 0\le n<m \),
\[
  \LL(x^n Q_m(x)) = 0.
\]
\end{corollary}

Since there are good bases as in \Cref{tab:1}, \Cref{thm:master}
implies the combinatorial models of generalized moments for ordinary
orthogonal polynomials, Laurent biorthogonal polynomials, and
orthogonal polynomials of types \( R_I \) and \( R_{II} \) in
Theorems~\ref{thm:mot}, \ref{thm:sch}, and \ref{thm:MS}.
\Cref{thm:master} also includes orthogonal polynomials that do not
belong to any of these four classes of orthogonal polynomials; for
example, one can consider the sequence \( \vc \) with \( c_{2n+1}=0 \)
and \( c_{2n}\ne0 \).

\begin{table}
  \centering
  \begin{center}
  \begin{tabular}{|c|c|} \hline
  \textbf{polynomials} & \textbf{good bases} \\ \hline
 {orthogonal polynomials} & \(  \{(n,m): m=0\} \) \\ \hline
 {Laurent biorthogonal polynomials} & \(  \{(n,m): n=0 \mbox{ or } m=0\} \) \\ \hline
 {orthogonal polynomials of type \( R_I \)} & \(  \{(n,m): n=0 \mbox{ or } m=0\} \) \\ \hline
  {orthogonal polynomials of type \( R_{II} \)} & \(  \{(n,m): 0\le n\le 1 \mbox{ or } m=0\} \) \\ \hline
\end{tabular}
\end{center}
\caption{Orthogonal polynomials and their good bases.}
\label{tab:1}
\end{table}

Note that in \Cref{thm:master} we have the assumption that \( W \) has
a good basis. It would be interesting to characterize the existence of
a good basis.

\begin{problem}
  Find a characterization of the sequences \( \va \), \( \vb \),
  \( \vc \), and \( \vl \) such that \( W \) has a good basis.
\end{problem}

\section{A combinatorial interpretation for dual coefficients}
\label{subsec:Comb}

In this section, we provide a combinatorial interpretation of the
(generalized) dual coefficients \( \tau_{n,r,s} \) of \( R_{II} \)
polynomials introduced in \Cref{sec:gener-orth-polyn}, which are
orthogonal polynomials of type \( R_{II} \) with relaxed conditions.

We start with an observation on the combinatorial interpretation for
generalized moments \( \mu_{n,r,s} \) of \( R_{II} \) polynomials. Let
\( P_n(x) \) be \( R_{II} \) polynomials defined in the previous
section. Recall that dual coefficients \( \tau_{n,r,s} \) are defined
by
\begin{equation}\label{eq:1}
  x^nP_r(x)=\sum_{s\ge0}\tau_{n,r,s}P_s(x).
\end{equation}
  Multiplying both sides of \eqref{eq:1} by \( Q_\ell(x) \) and
  applying \( \cL \) yields
\begin{equation}\label{eq:2}
  \mu_{n,r,\ell} = \sum_{s\ge0}\tau_{n,r,s}\mu_{0,s,\ell}.
\end{equation}

By \Cref{thm:master}, \( \mu_{n,r,\ell} \) is the generating function
for all \( R_{II} \) paths from \( (0,r) \) to \( (n,\ell) \). Hence, the
identity \eqref{eq:2} naturally suggests that \( \tau_{n,r,s} \) is
the generating function for the restricted \( R_{II} \) paths from
\( (0,r) \) to \( (n,s) \) defined in \Cref{def:rgm}. Indeed, such a
generating function satisfies \eqref{eq:2}. For a fixed \(\ell\),
however, this single identity does not determine all the coefficients
\(\tau_{n,r,s}\). For example, for \(n=r=\ell=0\), let \(F\) be any
formal power series and set
\[
  \tau'_{0,0,0}=1+F\mu_{0,1,0},\qquad
  \tau'_{0,0,1}=-F\mu_{0,0,0},\qquad
  \tau'_{0,0,s}=0\quad(s\ge2).
\]
Then we also have
\[
  \mu_{0,0,0} = \sum_{s\ge0}\tau'_{0,0,s}\mu_{0,s,0}.
\]
The suggested generating function nevertheless turns out to be equal
to \(\tau_{n,r,s}\), which is the main result in this section.

\begin{theorem}\label{thm:tau_nrs}
  For nonnegative integers \( n \), \( r \), and \( s \), we have
  \begin{equation}\label{eq:17}
   \tau_{n,r,s} = \sum_{p\in\rgm_{n,r,s}}\wt(p).
  \end{equation}
\end{theorem}

Note that in order to prove \Cref{thm:tau_nrs}, it suffices to show
that the right-hand side of \eqref{eq:17}, denoted
\( \tau'_{n,r,s} \), satisfies the same relation \eqref{eq:1} as
\( \tau_{n,r,s} \), that is,
\begin{equation}\label{eq:14}
  x^nP_r(x) = \sum_{s\ge0}\tau'_{n,r,s}P_s(x).
\end{equation}
We will prove \eqref{eq:14} by giving combinatorial interpretations
for both sides and finding a sign-reversing involution. To do this we
introduce the following definition.

\begin{definition}\label{def:tile}
  Consider a \( 1\times n \) square board whose boxes are labeled
  \( 1 \) through \( n \) from left to right. Let \( M_i \)
  (resp.~\( D_i \)) be a monomino (resp.~domino) with \( i \) dots
  inside. A \emph{dotted tiling of size \( n \)} is a tiling of the
  board using tiles in \( \{M_0,M_1,D_0,D_1,D_{2}\} \). The set of
  dotted tilings of size \( n \) is denoted by \( \dt_n \).

  For a dotted tiling \( T \), we define the weight \( \wt(T) \) to be
  the product of the weights \( \wt(\tau) \) of the tiles \( \tau \) in
  \( T \), where
\[
  \wt(\tau)=\begin{cases}
             x & \mbox{if } \tau=M_1,\\
             -b_{i-1} & \mbox{if } \tau=M_0 \mbox{ with label \( i \)},\\
             -\lambda_{i-1} & \mbox{if } \tau=D_0 \mbox{ with label \( (i-1,i) \)},\\
             -a_{i-1}x & \mbox{if } \tau=D_1 \mbox{ with label \( (i-1,i) \)},\\
             -c_{i-1}x^2 & \mbox{if } \tau=D_2 \mbox{ with label \( (i-1,i) \)}.
            \end{cases}
\]
See \Cref{fig:dotted_tiling} for an example.
\end{definition}

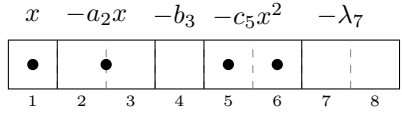
\begin{figure}
  \centering
  \begin{tikzpicture}[scale=1.3]
    \draw (0,0) rectangle (4,.5);
    \foreach \i in {1,...,7}{
      \draw[very thin, gray, dashed] (\i/2,0) -- (\i/2,.5);
    }

    {\tiny
    \foreach \i in {1,...,8}
    {
      \node at (\i/2-.25,-.12){\i};
    }
    }
    \draw (.5,0) -- (.5,.5);
    \draw (1.5,0) -- (1.5,.5);
    \draw (2,0) -- (2,.5);
    \draw (3,0) -- (3,.5);
    
    \node at (.25,.25)[circle,fill,inner sep=1.5pt]{};
    \node at (1,.25)[circle,fill,inner sep=1.5pt]{};
    \node at (2.25,.25)[circle,fill,inner sep=1.5pt]{};
    \node at (2.75,.25)[circle,fill,inner sep=1.5pt]{};

    \node at (.25,.75){\( x \)};
    \node at (.9,.75){\( -a_2x \)};
    \node at (1.7,.75){\( -b_3 \)};
    \node at (2.45,.75){\( -c_5x^2 \)};
    \node at (3.4,.75){\( -\lambda_7 \)};
  \end{tikzpicture}
  \caption{An example of \( T\in\dt_{8} \) with \( \wt(T)=a_2b_3c_5\lambda_7\,x^4. \)}
  \label{fig:dotted_tiling}
\end{figure}

By~\eqref{eq:three_term},
it is immediate that
  \begin{equation}\label{eq:18}
   P_n(x) = \sum_{T\in\dt_{n}}\wt(T). 
  \end{equation}
  We are now ready to prove \Cref{thm:tau_nrs}.

\begin{proof}[Proof of \Cref{thm:tau_nrs}]
  As discussed before, it suffices to show \eqref{eq:14}.
  By~\eqref{eq:18}, we can rewrite \eqref{eq:14} as
  \begin{equation}\label{eq:15}
   \sum_{k\ge0}\tau'_{n,r,k}P_k(x) = \sum_{(p,T)\in X} \wt(p)\wt(T),
  \end{equation}
  where \( X \) is the set of pairs \( (p,T) \) such that
  \( p\in \rgm_{n,r,k} \) and \( T\in \dt_k \) for some \( k\ge0 \).
  To prove \eqref{eq:14}, we establish a sign-reversing involution on
  \( X \) such that the sum over its fixed points is equal to
  \( x^nP_r(x) \).

  Consider \( (p,T)\in X \) with \( p\in \rgm_{n,r,k} \) and
  \( T\in \dt_k \). For \( m\ge1 \), we denote by \( p^m \) the
  \( m \)th step from the end of \( p \), and by \( T^m \) the
  \( m \)th tile from the end of \( T \). Let \( i \) (resp.~\( j \)) be
  the maximum number of consecutive \( U \)'s (resp.~\( M_1 \)'s) at
  the end of \( p \) (resp.~\( T \)). Here, if \( p^1\ne U \), then
  \( i=0 \), and if \( T^1\ne M_1 \), then \( j=0 \). We define
  \( \phi(p,T) = (p',T') \) considering the two cases \( i \le j \)
  and \( i > j \).
  
  \medskip

  \noindent
  \textbf{Case 1: \( i \le j \).} Note that \( p^{i+1} \) does not
  exist if and only if \( p \) consists only of up steps, in which
  case \( i=n=k-r \). We define \( p' \) and \( T' \) in the following
  subcases. In Cases 1-1 and 1-2, if \(T^{i+1}\) does not exist,
  we insert the new tile at the left end
  of \(T\).
  \begin{description}
  \item[Case 1-1] \( p^{i+1} = H \). Let \( p' \) be the path obtained from
    \( p \) by replacing \( p^{i+1} \) with \( U \). Let \( T' \) be the
    dotted tiling obtained from \( T \) by inserting a new tile \( M_0 \)
    after \( T^{i+1} \). 
  \item[Case 1-2] \( p^{i+1} = D \). Let \( p' \) be the path obtained from
    \( p \) by replacing \( p^{i+1} \) with \( U \). Let \( T' \) be the
    dotted tiling obtained from \( T \) by inserting a new tile \( D_0 \)
    after \( T^{i+1} \). 
  \item[Case 1-3] \( p^{i+1} = V \). In this case, we have \( i\ge1 \). Let
    \( p' \) be the path obtained from \( p \) by removing
    \( p^{i+1} \). Let \( T' \) be the dotted tiling obtained from \( T \) by
    replacing \( T^i \) with \( D_1 \).
  \item[Case 1-4] \( p^{i+1} = B \). In this case, we have
    \( i\ge2 \). Let \( p' \) be the path obtained from \( p \) by
    removing the pair of steps \( (p^{i+1},U) \), where \( U \) is the
    step following \( p^{i+1} \). Let \( T' \) be the dotted tiling obtained
    from \( T \) by replacing the pair of tiles \( (T^i,T^{i-1}) \)
    with a single tile \( D_2 \).
  \item[Case 1-5] \( p^{i+1} \) does not exist. In this case, let
    \( (p',T') = (p,T) \).
  \end{description}

  \medskip
  \noindent
  \textbf{Case 2: \( i > j \).} Then \( p^{j+1} \) is an up step
  starting at height \( k-j-1 \). If \( T^{j+1} \) is a domino, it
  occupies the boxes labeled \( k-j-1 \) and \( k-j \), so this
  height is at least \( 1 \). We define \( p' \) and \( T' \) in
  the following subcases.
  \begin{description}
  \item[Case 2-1] \( T^{j+1}=M_0 \). Let \( p' \) be the path obtained from
    \( p \) by replacing \( p^{j+1} \) with \( H \). Let \( T' \) be the
    dotted tiling obtained from \( T \) by removing the tile \( T^{j+1} \).
  \item[Case 2-2] \( T^{j+1}=D_0 \). Let \( p' \) be the path obtained from
    \( p \) by replacing \( p^{j+1} \) with \( D \). Let \( T' \) be the
    dotted tiling obtained from \( T \) by removing the tile \( T^{j+1} \).
  \item[Case 2-3] \( T^{j+1}=D_1 \). Let \( p' \) be the path obtained from
    \( p \) by inserting a step \( V \) before \( p^{j+1} \). Let \( T' \)
    be the dotted tiling obtained from \( T \) by replacing \( T^{j+1} \) with
    \( M_1 \).
  \item[Case 2-4] \( T^{j+1}=D_2 \). Let \( p' \) be the path obtained from
    \( p \) by inserting a pair of steps \( (B,U) \) before \( p^{j+1} \). Let \( T' \)
    be the dotted tiling obtained from \( T \) by replacing \( T^{j+1} \) with
   a pair of tiles \( (M_1,M_1) \).
  \end{description}

  For \( h=1,2,3,4 \), the operations in Cases 1-\(h\) and
  2-\(h\) are inverse to each other.

  See~\Cref{fig:sign_reversing_inv_1,fig:sign_reversing_inv_2,fig:sign_reversing_inv_3,fig:sign_reversing_inv_4}.
  It is straightforward to check that \( \phi \) is a sign-reversing
  involution on \( X \) whose fixed points are the pairs \( (p,T) \),
  where \( p \) is the unique path in \( \gm_{n,r,n+r} \) consisting
  only of up steps and \( T \) is a dotted tiling in \( \dt_{n+r} \) such
  that the last \( n \) tiles are all equal to \( M_1 \). Since
  \( \wt(p)\wt(T) = x^n \wt(T_0) \), where \( T_0 \) is the dotted tiling
  obtained from \( T \) by removing the last \( n \) tiles, we have
  \begin{equation}\label{eq:16}
    \sum_{(p,T)\in X} \wt(p)\wt(T)
    = x^n \sum_{T_0\in \dt_r} \wt(T_0) = x^n P_r(x).
  \end{equation}
  Then \eqref{eq:14} follows from \eqref{eq:15} and \eqref{eq:16},
  which completes the proof.
\end{proof}

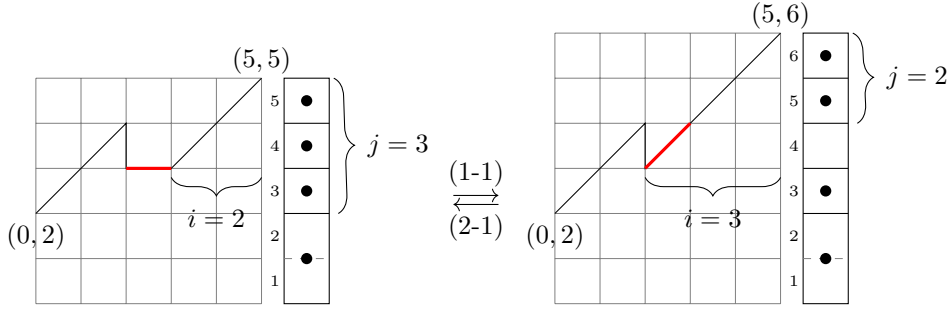
\begin{figure}
  \centering
  \begin{tikzpicture}[scale=0.6]
    \draw[help lines] (0,0) grid (5,5);

    \node at (0,1.5) {\( (0,2) \)};
    \node at (5,5.4) {\( (5,5) \)};
    
    \draw [decorate,decoration={brace,mirror,amplitude=7}] (3,2.8) --node[below=3mm]{\( i=2 \)} (5,2.8);
    \draw  (0,2) -- (1,3) -- (2,4) -- (2,3) -- (3,3) -- (5,5);
    \draw[very thick, red] (2,3) -- (3,3);

    \begin{scope}[shift={(5.5,0)}]  
    \draw (0,0) rectangle (1,5);
    \foreach \i in {1,...,5} {
        \node at (-.2,-.5+\i){{\tiny\( \i \)}};
    }
    \foreach \i in {2,...,4}{
      \draw (0,\i) -- (1,\i);
      \node[circle,fill,inner sep=1.5pt] at (.5, \i+.5){};
    }
    \draw[very thin, gray, dashed] (0,1) -- (1,1){};

    \node[circle,fill,inner sep=1.5pt] at (.5,1){};
    \draw [decorate,rotate = 90,decoration={brace,mirror,amplitude=7}] (2,-1.2) --node[right=3mm]{\( j=3 \)} (5,-1.1){};
    \end{scope}

    \begin{scope}[shift={(11.5,0)}]
      \draw[->] (-2.25,2.4) -- node[above]{(1-1)} (-1.25,2.4); 
      \draw[<-] (-2.25,2.2) -- node[below]{(2-1)} (-1.25,2.2);

      \draw[help lines] (0,0) grid (5,6);
      \node at (0,1.5) {\( (0,2) \)};
      \node at (5,6.4) {\( (5,6) \)};
      
      \draw [decorate,decoration={brace,mirror,amplitude=7}] (2,2.8) --node[below=3mm]{\( i=3 \)} (5,2.8);
      \draw  (0,2) -- (1,3) -- (2,4) -- (2,3) -- (5,6);
      \draw[very thick, red] (2,3) -- (3,4);
      
      \begin{scope}[shift={(5.5,0)}]  
        
        \draw (0,0) rectangle (1,6);
        \foreach \i in {1,...,6} {
          \node at (-.2,-.5+\i){{\tiny\( \i \)}};
        }
        \foreach \i in {2,...,5}{
          \draw (0,\i) -- (1,\i);
        }
        \draw[very thin, gray, dashed] (0,1) -- (1,1){};
        
        \node[circle,fill,inner sep=1.5pt] at (.5,1){};
        \node[circle,fill,inner sep=1.5pt] at (.5,2.5){};
        \node[circle,fill,inner sep=1.5pt] at (.5,4.5){};
        \node[circle,fill,inner sep=1.5pt] at (.5,5.5){};
        \draw [decorate,rotate = 90,decoration={brace,mirror,amplitude=7}] (4,-1.2) --node[right=3mm]{\( j=2 \)} (6,-1.1){};
      \end{scope}
    \end{scope}
  \end{tikzpicture}
  \caption{An example of \( (P,T)\in \rgm_{5,2,5}\times \dt_{5} \) for Case 1-1 and its image \( (\phi(P),\phi(T))\in \rgm_{5,2,6}\times \dt_{6}\) corresponding to Case 2-1, described in the proof of~\Cref{thm:tau_nrs}. }
  \label{fig:sign_reversing_inv_1}
\end{figure}

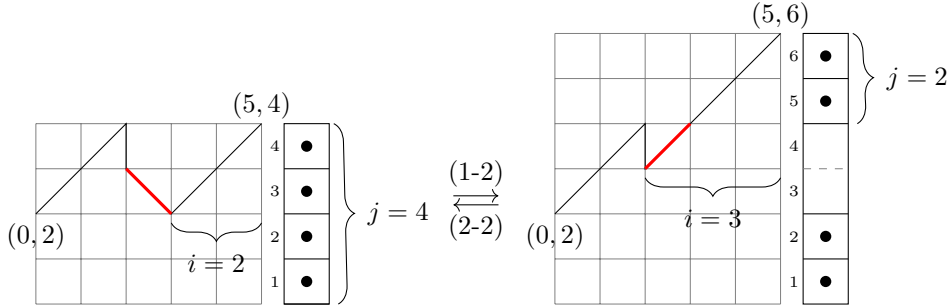
\begin{figure}
  \centering
  \begin{tikzpicture}[scale=0.6]
    \draw[help lines] (0,0) grid (5,4);
    
    \node at (0,1.5) {\( (0,2) \)};
    \node at (5,4.4) {\( (5,4) \)};
    
    
    \draw [decorate,decoration={brace,mirror,amplitude=7}] (3,1.8) --node[below=3mm]{\( i=2 \)} (5,1.8);
    \draw  (0,2) -- (1,3) -- (2,4) -- (2,3) -- (3,2) -- (5,4);
    \draw[very thick, red] (2,3) -- (3,2);

    \begin{scope}[shift={(5.5,0)}]
      
      \draw (0,0) rectangle (1,4);
      \foreach \i in {1,...,4} {
        \node at (-.2,-.5+\i){{\tiny\( \i \)}};
      }
      \foreach \i in {0,...,3}{
        \draw (0,\i) -- (1,\i);
      \node[circle,fill,inner sep=1.5pt] at (.5, \i+.5){};
    }
    \draw [decorate,rotate = 90,decoration={brace,mirror,amplitude=7}] (0,-1.2) --node[right=3mm]{\( j=4 \)} (4,-1.1){};
  \end{scope}
  
  \begin{scope}[shift={(11.5,0)}]
    
    \draw[help lines] (0,0) grid (5,6);
    \draw[->] (-2.25,2.4) -- node[above]{(1-2)} (-1.25,2.4); 
    \draw[<-] (-2.25,2.2) -- node[below]{(2-2)} (-1.25,2.2);

    \node at (0,1.5) {\( (0,2) \)};
    \node at (5,6.4) {\( (5,6) \)};

    \draw [decorate,decoration={brace,mirror,amplitude=7}] (2,2.8) --node[below=3mm]{\( i=3 \)} (5,2.8);
    \draw  (0,2) -- (1,3) -- (2,4) -- (2,3) -- (5,6);
    \draw[very thick, red] (2,3) -- (3,4);
    
    \begin{scope}[shift={(5.5,0)}]  
     
      \draw (0,0) rectangle (1,6);
      \foreach \i in {1,...,6} {
        \node at (-.2,-.5+\i){{\tiny\( \i \)}};
      }
        \foreach \i in {1,2,4,5}{
          \draw (0,\i) -- (1,\i);
        }
        \draw[very thin, gray, dashed] (0,3) -- (1,3){};
        
        \foreach \i in {0,1,4,5}{
          \node[circle,fill,inner sep=1.5pt] at (.5, \i+.5){};
        }
        
        \draw [decorate,rotate = 90,decoration={brace,mirror,amplitude=7}] (4,-1.2) --node[right=3mm]{\( j=2 \)} (6,-1.1){};
      \end{scope}
    \end{scope}
  \end{tikzpicture}
  \caption{An example of \( (P,T)\in \rgm_{5,2,4}\times \dt_{4} \) for Case 1-2 and its image \( (\phi(P),\phi(T))\in \rgm_{5,2,6}\times \dt_{6}\) corresponding to Case 2-2, described in the proof of~\Cref{thm:tau_nrs}.}
  \label{fig:sign_reversing_inv_2}
\end{figure}

\begin{figure}
  \centering
  \begin{tikzpicture}[scale=0.6]
    \draw[help lines] (0,0) grid (5,4);
        
    \node at (0,1.5) {\( (0,2) \)};
    \node at (5,4.4) {\( (5,4) \)};
    
    \draw [decorate,decoration={brace,mirror,amplitude=7}] (3,1.8) --node[below=3mm]{\( i=2 \)} (5,1.8);
    \draw  (0,2) -- (1,3) -- (2,4) -- (2,3) -- (3,3) -- (3,2) -- (5,4);
    \draw[very thick, red] (3,3) -- (3,2);
    
    \begin{scope}[shift={(5.5,0)}]  
      \draw (0,0) rectangle (1,4);
      \foreach \i in {1,...,4} {
        \node at (-.2,-.5+\i){{\tiny\( \i \)}};
      }
      \foreach \i in {0,...,3}{
        \draw (0,\i) -- (1,\i);
      }
      
      \node[circle,fill,inner sep=1.5pt] at (.5, 2.5){};
      \node[circle,fill,inner sep=1.5pt] at (.5, 3.5){};
      
      \draw [decorate,rotate = 90,decoration={brace,mirror,amplitude=7}] (2,-1.2) --node[right=3mm]{\( j=2 \)} (4,-1.1){};
    \end{scope}
    
    \begin{scope}[shift={(11.5,0)}]
      \draw[->] (-2.25,2.4) -- node[above]{(1-3)} (-1.25,2.4); 
      \draw[<-] (-2.25,2.2) -- node[below]{(2-3)} (-1.25,2.2);
      
      \draw[help lines] (0,0) grid (5,5);
      \node at (0,1.5) {\( (0,2) \)};
      \node at (5,5.4) {\( (5,5) \)};
      
      \draw [decorate,decoration={brace,mirror,amplitude=7}] (3,2.8) --node[below=3mm]{\( i=2 \)} (5,2.8);
      \draw  (0,2) -- (1,3) -- (2,4) -- (2,3) -- (3,3) -- (5,5);
      \node[circle, fill, color = red, inner sep=1pt] at (3,3){};
      
      \begin{scope}[shift={(5.5,0)}]  
        \draw (0,0) rectangle (1,5);
        \foreach \i in {1,...,5} {
          \node at (-.2,-.5+\i){{\tiny\( \i \)}};
        }
        \foreach \i in {1,2,4}{
          \draw (0,\i) -- (1,\i);
        }
        \draw[very thin, gray, dashed] (0,3) -- (1,3){};
        
        \node[circle,fill,inner sep=1.5pt] at (.5, 4.5){};
        \node[circle,fill,inner sep=1.5pt] at (.5, 3){};
        
        \draw [decorate,rotate = 90,decoration={brace,mirror,amplitude=7}] (4,-1.2) --node[right=3mm]{\( j=1 \)} (5,-1.1){};
      \end{scope}
    \end{scope}
  \end{tikzpicture}
  \caption{An example of \( (P,T)\in \rgm_{5,2,4}\times \dt_{4} \) for Case 1-3 and its image \( (\phi(P),\phi(T))\in\rgm_{5,2,5}\times \dt_{5}\) corresponding to Case 2-3, described in the proof of~\Cref{thm:tau_nrs}.}
  \label{fig:sign_reversing_inv_3}
\end{figure}
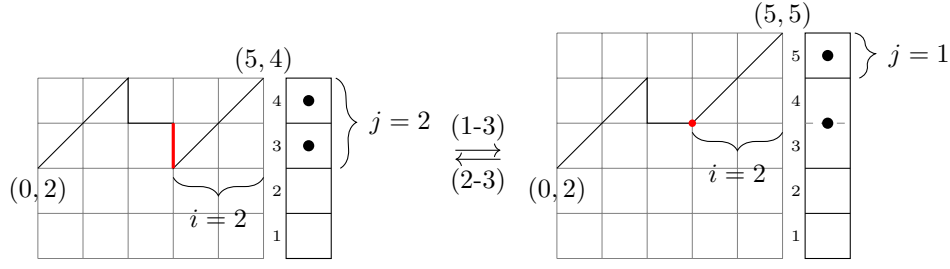

\begin{figure}
  \tikzset{
    set arrow inside/.code={\pgfqkeys{/tikz/arrow inside}{#1}},
    set arrow inside={end/.initial=>, opt/.initial=},
    /pgf/decoration/Mark/.style={
        mark/.expanded=at position #1 with
        {
          \noexpand\arrow[\pgfkeysvalueof{/tikz/arrow inside/opt}]{\pgfkeysvalueof{/tikz/arrow inside/end}}
        }
    },
    arrow inside/.style 2 args={
        set arrow inside={#1},
        postaction={
            decorate,decoration={
                markings,Mark/.list={#2}
            }
        }
    },
}
  \centering
  \begin{tikzpicture}[scale=0.6]
    \draw[help lines] (0,0) grid (5,4);
        
    \node at (0,1.5) {\( (0,2) \)};
    \node at (5,4.4) {\( (5,4) \)};

    \draw[thick, red] (3,2) -- (4,3);
    \draw[thick, red, arrow inside={end=stealth}{0.57}] (4,3) to[bend right=10] (3,2);

    \draw [decorate,decoration={brace,mirror,amplitude=7}] (3,1.8) --node[below=3mm]{\( i=2 \)} (5,1.8);
    \draw  (0,2) -- (1,3) -- (2,4) -- (2,3) -- (4,3) -- (5,4);
    
    \begin{scope}[shift={(5.5,0)}]  
      \draw (0,0) rectangle (1,4);
      \foreach \i in {1,...,4} {
        \node at (-.2,-.5+\i){{\tiny\( \i \)}};
      }
      \foreach \i in {0,...,3}{
        \draw (0,\i) -- (1,\i);
      }
      
      \node[circle,fill,inner sep=1.5pt] at (.5, 2.5){};
      \node[circle,fill,inner sep=1.5pt] at (.5, 3.5){};
      
      \draw [decorate,rotate = 90,decoration={brace,mirror,amplitude=7}] (2,-1.2) --node[right=3mm]{\( j=2 \)} (4,-1.1){};
    \end{scope}
    
    \begin{scope}[shift={(11.5,0)}]
      \draw[->] (-2.25,2.4) -- node[above]{(1-4)} (-1.25,2.4); 
      \draw[<-] (-2.25,2.2) -- node[below]{(2-4)} (-1.25,2.2);
      
   \draw[help lines] (0,0) grid (5,4);
        
    \node at (0,1.5) {\( (0,2) \)};
    \node at (5,4.4) {\( (5,4) \)};

    \draw [decorate,decoration={brace,mirror,amplitude=7}] (4,1.8) --node[below=3mm]{\( i=1 \)} (5,1.8);
    \draw  (0,2) -- (1,3) -- (2,4) -- (2,3) -- (4,3) -- (5,4);
    \node[circle, fill, color = red, inner sep=1pt] at (4,3){};
    
    \begin{scope}[shift={(5.5,0)}]  
      \draw (0,0) rectangle (1,4);
      \foreach \i in {1,...,4} {
        \node at (-.2,-.5+\i){{\tiny\( \i \)}};
      }
      \foreach \i in {0,...,2}{
        \draw (0,\i) -- (1,\i);
      }
      
      \draw[very thin, gray, dashed] (0,3) -- (1,3){};

      \node[circle,fill,inner sep=1.5pt] at (.5, 2.5){};
      \node[circle,fill,inner sep=1.5pt] at (.5, 3.5){};
      
     \node[right] at (1,4){\( j=0 \)};
    \end{scope}
    \end{scope}
  \end{tikzpicture}
  \caption{An example of \( (P,T)\in \rgm_{5,2,4}\times \dt_{4} \) for Case 1-4 and its image \( (\phi(P),\phi(T))\in\rgm_{5,2,4}\times \dt_{4}\) corresponding to Case 2-4, described in the proof of~\Cref{thm:tau_nrs}.}
  \label{fig:sign_reversing_inv_4}
\end{figure}
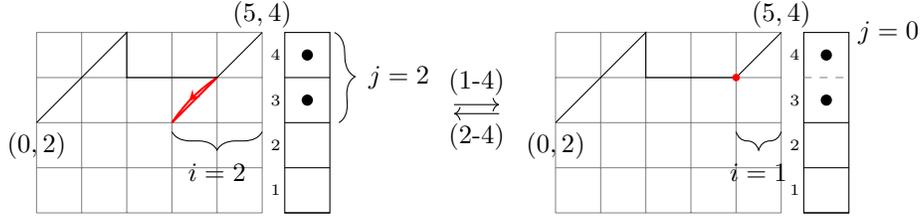

If \( c_n = 0 \) in \Cref{thm:tau_nrs}, we obtain the following
corollary.

\begin{corollary}
  Let \( (P_n^I(x))_{n\ge0} \) be the sequence of orthogonal
  polynomials of type \( R_I \) defined in \eqref{eq:3rr-1} and let
  \( \tau_{n,r,s}^I \) be their dual coefficients. Then we have
  \[
   \tau^I_{n,r,s} = \sum_{p\in\rgm^I_{n,r,s}}\wt(p),
  \]
  where \( \rgm^I_{n,r,s} \) is the set of restricted \( R_{II} \)
  paths from \( (0,r) \) to \( (n,s) \) without backward-down steps.
\end{corollary}

Recall that a Motzkin path is an \( R_{II} \) path without
vertical-down steps and backward-down steps. Thus, there is no
difference between Motzkin paths and restricted Motzkin paths.
Therefore, if \( c_n = a_n=0 \) in \Cref{thm:tau_nrs}, we obtain
\Cref{cor:mot2}.

\begin{remark}
  One may try to prove \Cref{thm:tau_nrs} by induction on \( n \).
  Once the base steps \( n=0 \) and \( n=1 \) have been proved, we can
  easily apply the inductive argument as in the case of
  \( \mu_{n,r,s} \). The case \( n=0 \) is immediate from the
  definition. However, the other case \( n=1 \) is not obvious. Note
  that, for the case of \( \mu_{n,r,s} \), both cases \( n=0 \) and
  \( n=1 \) follow from the definition of \( \LL \).
\end{remark}

\section{Convergence of the moments and dual coefficients} 
\label{subsec:conv}

In \Cref{sec:comb-interpr-gener}, we found a combinatorial description
for generalized moments \( \mu_{n,r,s} \) in terms of the weight sum
of \( R_{II} \) paths. As there are infinitely many \( R_{II} \) paths
from \( (0,r) \) to \( (n,s) \), we need to check whether the weight
sum converges. In this section, we give sufficient conditions on the
weights of steps for convergent weight sums of \( R_{II} \) paths and
restricted \( R_{II} \) paths.

A \emph{Dyck path of semilength \( n \)} is a lattice path from
\( (0,0) \) to \( (2n,0) \) with steps \( U=(1,1) \) and
\( D=(1,-1) \) that never goes below the \( x \)-axis. Such paths are
enumerated by the Catalan number
\( C_n=\frac{1}{n+1}\binom{2n}{n} \); see, for example,
\cite[Section~1.1]{Stanley2015}.

\begin{lemma}\label{lem:fixed_weight_mu}
  Suppose that \(b_m=b\) for \(m\ge0\), and that \(a_m=a\),
  \(c_m=c\), and \(\lambda_m=\lambda\) for \(m\ge1\). Then
  \[
    \mu_{n}= \sum_{\substack{i,j,k\ge0\\2i+j\le n}} C_{i+j+k}
    \binom{n+j+2k}{n-2i-j} \binom{i+j+k}{j} \binom{i+k}{i}
    b^{n-2i-j} \lambda^i a^j c^k.
  \]
\end{lemma}
\begin{proof}
  Suppose that \( p\in \gm_n \) has \( u \) up steps, \( h \)
  horizontal steps, \( i \) down steps, \( j \) vertical-down steps,
  and \( k \) backward-down steps. Then
  \[
    (n,0) = u(1,1) + h(1,0) + i(1,-1) +j(0,-1) + k(-1,-1),
  \]
  which implies \( u=i+j+k \) and therefore \( h=n-2i-j \). Consequently
  \begin{equation}\label{eq:11}
    \mu_{n}= \sum_{p\in \gm_n} \wt(p)
    = \sum_{\substack{i,j,k\ge0\\2i+j\le n}} f_n(i,j,k) b^{n-2i-j} \lambda^i a^j c^k,
  \end{equation}
  where \( f_n(i,j,k) \) is the number of paths \( p\in \gm_n \) with
  \( i \) down steps, \( j \) vertical-down steps, and \( k \)
  backward-down steps.

  Given a path \( p\in \gm_n \), define \( \phi(p) \) to be the
  labeled Motzkin path obtained from \( p \) by replacing every
  vertical-down step with a down step labeled \( a \), every backward-down
  step with a down step labeled \( c \). If \( p \) has the given
  number of steps of each type as above, then \( \phi(p) \) is a
  labeled Motzkin path from \( (0,0) \) to \( (n+j+2k,0) \) with
  \( n-2i-j \) horizontal steps, \( i \) (unlabeled) down steps,
  \( j \) down steps labeled \( a \), and \( k \) down steps labeled
  \( c \). Such a labeled Motzkin path can be obtained from a Dyck
  path from \( (0,0) \) to \( (2i+2j+2k,0) \) by inserting \( n-2i-j \)
  horizontal steps in \( \binom{n+j+2k}{n-2i-j} \) ways, labeling
  \( j \) down steps by \( a \) in \( \binom{i+j+k}{j} \) ways, and
  labeling \( k \) down steps by \( c \) in
  \( \binom{i+k}{k} = \binom{i+k}{i} \) ways. This implies that
  \begin{equation}\label{eq:12}
    f_n(i,j,k) = C_{i+j+k} \binom{n+j+2k}{n-2i-j}
    \binom{i+j+k}{j} \binom{i+k}{i}.
  \end{equation}
  By \eqref{eq:11} and \eqref{eq:12}, we obtain the desired formula.
\end{proof}
\begin{lemma}\label{lem:ratio_test}
  Suppose that \(b_m=b\) for \(m\ge0\), and that \(a_m=a\),
  \(c_m=c\), and \(\lambda_m=\lambda\) for \(m\ge1\), with
  \(|c|<1/4\). Then \(\mu_n\)
  converges absolutely.
\end{lemma}

\begin{proof}
  By \Cref{lem:fixed_weight_mu}, we have
  \( \mu_n=\sum_{\substack{i,j,k\ge0\\2i+j\le n}} f_n(i,j,k) b^{n-2i-j} \lambda^i a^j c^k \), where \( f_n(i,j,k) \)
  is given in \eqref{eq:12}. There are finitely many pairs \( (i,j) \)
  in this sum. Thus it suffices to show that,
  for fixed \( i \) and \( j \) with \( 2i+j\le n \), the sum
  \begin{equation}\label{eq:13}
    \sum_{k\ge0} f_n(i,j,k) c^k
  \end{equation}
  converges.

  By Stirling's formula, for fixed \( m \) and large \( N \), we
  have the asymptotic behaviors
\[
C_N \sim \frac{4^N}{N^{3/2} \sqrt{\pi}}, \qquad 
\binom{N}{m} \sim \frac{N^m}{m!}.
\]
Thus, for fixed \( i,j,n \) and large \( k \),
\begin{align*}
  f_n(i,j,k) c^k
  &\sim \frac{4^{i+j+k}}{(i+j+k)^{3/2} \sqrt{\pi}}
    \cdot \frac{(n+j+2k)^{n-2i-j}}{(n-2i-j)!}
    \cdot \frac{(i+j+k)^j}{j!}
    \cdot \frac{(i+k)^i}{i!} c^k\\
  &\sim  D\cdot (4c)^k k^{n-i-3/2},
\end{align*}
where \( D \) is a constant with respect to \( k \). Since
\( |4c|<1 \), by the ratio test, \eqref{eq:13} converges.
\end{proof}

\begin{theorem}\label{thm:convergence}
  Let \( (a_m)_{m\ge0},(b_m)_{m\ge0}, (c_m)_{m\ge0} \) and \( (\lambda_m)_{m\ge0} \) be bounded sequences of complex numbers. Suppose 
  \( |c_m|<\frac{1}{4}-\epsilon \) for all \( m \), where
  \( \epsilon>0 \) is a fixed number. Then the generalized moments
  \( \mu_{n,r,s} \) converge absolutely.
\end{theorem}
\begin{proof}
  Choose positive numbers \( a,b,\lambda \) such that
  \( |a_m|<a \), \( |b_m|<b \), and \( |\lambda_m|<\lambda \) for
  all \( m \), and set \( R=\max\{\lambda,1\} \) and
  \( c=1/4-\epsilon \). Let \( \widehat{\wt}(p) \) be the weight of
  \( p \) after the substitution
  \( (b_i,\lambda_i,a_i,c_i)\mapsto(b,R,a,c) \) for all \( i \).
  Then \( |\wt(p)|\le\widehat{\wt}(p) \).

  For \( p\in \gm_{n,r,s} \), let \( p' \) be the \( R_{II} \) path
  obtained from \( p \) by adding \( r \) up steps at the beginning
  and \( s \) down steps at the end. Then \( p'\in \gm_{n+r+s} \).
  Since the map \( p\mapsto p' \) is injective and
  \( \widehat{\wt}(p)=R^{-s}\widehat{\wt}(p') \le\widehat{\wt}(p') \),
  we have
  \[
    \sum_{p\in \gm_{n,r,s}} |\wt(p)|
    \le \sum_{p'\in \gm_{n+r+s}} \widehat{\wt}(p').
  \]
  By \Cref{lem:ratio_test}, the right-hand side converges, which
  implies that \( \mu_{n,r,s} \) converges absolutely.
\end{proof}

For the convergence of \( \tau_{n,r,s} \), we only need the condition
on \( \vc \).

\begin{corollary}\label{cor:conv_tau_nrs}
  Suppose that \( |c_m|<\frac{1}{4}-\epsilon \) for all \( m \), where
  \( \epsilon>0 \) is a fixed number. Then the dual coefficients
  \( \tau_{n,r,s} \) converge absolutely.
\end{corollary}
\begin{proof}
  Recall that \( \tau_{n,r,s} = \sum_{p\in \rgm_{n,r,s}}\wt(p) \).
  Note that if \( (x,y) \) is a point in any \( R_{II} \)-path
  \( p\in \gm_{n,r,s} \), then \( y-x \) never increases along the
  path, which implies \( y-x\le r-0=r \). Moreover, if
  \( p\in \rgm_{n,r,s} \), then we also have \( x\le n \). Therefore,
  each \( p\in \rgm_{n,r,s} \) cannot go above the line \( y = n+r \).
  Hence
  \begin{equation}\label{eq:19}
    \sum_{p\in \rgm_{n,r,s}}|\wt(p)| \le \sum_{p\in \gm_{n,r,s}}|\wt'(p)|,
  \end{equation}
  where \( \wt'(p) \) is the weight of \( p \) using the sequences
  \( \va',\vb',\vc' \), and \( \vl' \) obtained from
  \( \va,\vb,\vc \), and \( \vl \) by substituting
  \( b_i = a_i = \lambda_i = c_i =0 \) for \( i > n+r \). Then,
  by~\Cref{thm:convergence}, we see that the right-hand side of
  \eqref{eq:19} converges.
\end{proof}

\section{Moments with constant recurrence coefficients}
\label{sec:moments-with-const}

In this section, we assume that \(b_m=b\) for \(m\ge0\), and that
\(a_m=a\), \(c_m=c\), and \(\lambda_m=\lambda\) for \(m\ge1\).  
The goal is to find a formula for \(\mu_n\) in this case by considering 
it as a formal power series in variables \( a,b,c,\lambda \).

Let
\[
  \mu(x)= \sum_{n\ge0} \mu_n x^n.
\]
Considering the first time each \( R_{II} \) path returns to the
\( x \)-axis as shown in \Cref{fig:visualization_mu_x}, we obtain the
following functional equation:
\[
  \mu(x)=1+bx\mu(x)+(c+ax+\lambda x^2)\mu(x)^2.
\]
\begin{figure}
  \centering
    \begin{tikzpicture}
    \begin{scope}[scale=0.38]
      \node at (-4,0) {\( \mu(x) = \)};
      
      \node at (-2,0) [circle, fill,inner sep=.5pt]{};
      \node at (-2,1) {\( 1 \)};
      
      \node at (-1,0) {\( + \)};
      
      \draw (0,0) -- node [above] {\( bx \)} (1,0);
      \draw (1,0) arc (180:0:1.5) -- cycle;
      \node at (2.5,0.6) {\( \mu(x) \)};

      \node at (5,0) {\( + \)};
      
      \draw (6,0) -- node [text width=0.3cm,above ] {\( x \)}  (7,1);
      \draw (7,1) arc (180:0:1.5) -- cycle;
      \node at (8.5,1.6) {\( \mu(x) \)};
      \draw (10,1) --  node [ left ] {\( \frac{c}{x} \)}  (9,0);
      \node at (11.5,1.3) {\( \mu(x) \)};
      \draw (9,0) arc (130:50:3) -- cycle;
      
    \end{scope}
    \begin{scope}[scale=0.38, xshift=9cm]
      \node at (5,0) {\( + \)};
      
      \draw (6,0) -- node [text width=0.3cm,above ] {\( x \)}  (7,1);
      \draw (7,1) arc (180:0:1.5) -- cycle;
      \node at (8.5,1.6) {\( \mu(x) \)};
      \draw (10,1) --  node [ left ] {\( a \)}  (10,0);
      \node at (11.5,0.6) {\( \mu(x) \)};
      \draw (10,0) arc (180:0:1.5) -- cycle;
    \end{scope}
\begin{scope}[scale=0.38, xshift=18cm]
      \node at (5,0) {\( + \)};
      
      \draw (6,0) -- node [text width=0.3cm,above ] {\( x \)}  (7,1);
      \draw (7,1) arc (180:0:1.5) -- cycle;
      \node at (8.5,1.6) {\( \mu(x) \)};
      \draw (10,1) --  node [ left ] {\( \lambda x \)}  (11,0);
      \node at (12.5,0.6) {\( \mu(x) \)};
      \draw (11,0) arc (180:0:1.5) -- cycle;
    \end{scope}

  \end{tikzpicture}
  \caption{A visualization of the functional equation for \( \mu(x) \).} 
  \label{fig:visualization_mu_x}
\end{figure}
Solving the equation yields
\begin{align*}
  \mu(x)
  &= \frac{1-bx-\sqrt{(1-4c)-(4a+2b)x-(4\lambda-b^2)x^2}}{2(c+ax+\lambda x^2)}\\
  &=\frac{2}{1-bx+\sqrt{(1-4c)-(4a+2b)x-(4\lambda-b^2)x^2}}.
\end{align*}

We aim to show that (up to a constant) \( \mu_n \) can be expressed as
the moment of certain ordinary or type \( R_I \) orthogonal
polynomials. Recall that the moments of ordinary orthogonal
polynomials (resp.~orthogonal polynomials of type \( R_I \)) are equal
to the weight sums of certain Motzkin paths (resp.~Motzkin--Schr\"oder
paths). Thus, the goal is to reinterpret the values \( \mu_n \) in the
framework of Motzkin or Motzkin--Schr\"oder path models. While there
are infinitely many \( R_{II} \) paths from \( (0,0) \) to
\( (n,0) \), the number of such Motzkin or Motzkin--Schr\"{o}der paths
is finite. From this point of view, our results give useful ways of
calculating \( \mu_n \). Note that this section generalizes the result
on moments of \( R_I \) polynomials in \cite[Section
6.2]{kimstanton:R1}.

Let \( \mathcal{C} \) be the generating function for the Catalan numbers
\( C_n \):
\[
  {\mathcal C} = \sum_{n\ge0} C_n c^n =
  \frac{1-\sqrt{1-4c}}{2c} = \frac{2}{1+\sqrt{1-4c}}.
\]

\begin{proposition}\label{prop:R2_Classical}
  Suppose that \(b_m=b\) for \(m\ge0\), and that \(a_m=a\),
  \(c_m=c\), and \(\lambda_m=\lambda\) for \(m\ge1\). We have
  \[
    \mu(x)={\mathcal C}\sum_{n\ge0}
    \mu_n(\vec B,\vec \Lambda)x^n.
  \]
  Equivalently, for every \( n\ge0 \),
  \[
    \mu_n={\mathcal C}\cdot \mu_n(\vec B,\vec \Lambda).
  \]
  Here, \( \mu_n(\vec B,\vec \Lambda) \) is the \( n \)th moment for
  ordinary orthogonal polynomials given by
\[ 
 p_{n+1}(x)=(x-B_n)p_n(x)-\Lambda_n p_{n-1}(x), n\ge0, 
\]
where \( p_{-1}(x)=0,p_0(x)=1 \), and
\begin{align*} 
  B_0&=\frac{\mathcal C a+b}{\sqrt{1-4c}}, & \Lambda_1&=\frac{\mathcal C(b^2c+a^2+ab-4c\lambda+\lambda)}{\sqrt{1-4c}^3}, \\
  B_n&=\frac{2a+b}{1-4c}, & \Lambda_{n+1}&=\frac{b^2c+a^2+ab-4c\lambda+\lambda}{(1-4c)^2}, \qquad n\ge1.
\end{align*}
\end{proposition}

\begin{proof}
  By Flajolet's theory on continued fractions \cite{Flajolet1980},
  \begin{equation}\label{eq:20}
    \sum_{n\ge0} \mu_n(\vec B,\vec \Lambda)x^n  = \cfrac{1}{1-B_0x- \Lambda_1x^2\cdot g(x)},
  \end{equation}
  where
  \[
    g(x) = \cfrac{1}{1-B_1x-\cfrac{\Lambda_2x^2}{1-B_2x-\cdots}}.
  \]
  Since \( B_1 = B_2 = \cdots \) and \( \Lambda_2=\Lambda_3=\cdots \), we see that
  \[
    g(x) = \cfrac{1}{1-B_1x-\Lambda_2x^2\cdot g(x)}.
  \]
  Solving this functional equation gives
\begin{equation}\label{eq:21}
    g(x) = \frac{2}{1-B_1x+\sqrt{1-2B_1x+(B_1^2-4\Lambda_2)x^2}}.
  \end{equation}
By \eqref{eq:20} and \eqref{eq:21},
\begin{align*}
  \mathcal{C} \sum_{n\ge0} \mu_n(\vec B,\vec \Lambda)x^n 
  &= \cfrac{\mathcal{C}}{1-B_0x-\cfrac{2\Lambda_1x^2}{1-B_1x+\sqrt{1-2B_1x+(B_1^2-4\Lambda_2)x^2}}} \\
  &= \frac{2}{1-bx+\sqrt{(1-4c)-(4a+2b)x-(4\lambda-b^2)x^2}}\\
  &= \mu(x),
\end{align*}
as desired.
\end{proof}

As a corollary, we obtain a formula for the Hankel determinant \( \det(\mu_{i+j})_{i,j=0}^n \).

\begin{corollary}\label{cor:Hankel}
  Suppose that \(b_m=b\) for \(m\ge0\), and that \(a_m=a\),
  \(c_m=c\), and \(\lambda_m=\lambda\) for \(m\ge1\). Then, we have
  \[
    \det(\mu_{i+j})_{i,j=0}^n= \frac{(b^2c+a^2+ab-4c\lambda+\lambda)^{\binom{n+1}{2}}\mathcal C^{2n+1}}{\sqrt{1-4c}^{\binom{2n+1}{2}}}.
  \]
\end{corollary}
\begin{proof}
  By \Cref{prop:R2_Classical} and the fact that
  \( \det(\mu_{i+j}(\vec B,\vec \Lambda))_{i,j=0}^n=  \Lambda_1^n\Lambda_{2}^{n-1}\cdots\Lambda_{n}^1\),
 we have 
\begin{align*} 
  \det(\mu_{i+j})_{i,j=0}^n
  &= \mathcal{C}^{n+1}\det(\mu_{i+j}(\vec B,\vec \Lambda))_{i,j=0}^n\\
  &=\mathcal{C}^{n+1}\Lambda_1^n\Lambda_{2}^{n-1}\cdots\Lambda_{n}^1\\ 
  &= \frac{(b^2c+a^2+ab-4c\lambda+\lambda)^{\binom{n+1}{2}} \mathcal C^{2n+1}}{\sqrt{1-4c}^{\binom{2n+1}{2}}}.\qedhere
\end{align*}
\end{proof}

Now we show that (up to a constant) \( \mu_n\) can also be interpreted
as the moment of orthogonal polynomials of type \( R_I \).

\begin{proposition}\label{prop:R2_R1}
  Suppose that \(b_m=b\) for \(m\ge0\), and that \(a_m=a\),
  \(c_m=c\), and \(\lambda_m=\lambda\) for \(m\ge1\). Then
  \[
    \mu_n = \mathcal{C}\cdot \mu_n(\vec B,\vec \Lambda,\vec A).
  \]
    Here, \( \mu_n(\vec B,\vec \Lambda,\vec A) \) is the \( n \)th moment for
  orthogonal polynomials of type \( R_I \) given by
\[ 
 p_{n+1}(x)=(x-B_n)p_n(x)-(A_nx+\Lambda_n) p_{n-1}(x), \qquad  n\ge0, 
\]
where  \( p_{-1}(x)=0,p_0(x)=1 \), and
\begin{align*} 
  B_0&=\frac{b}{\sqrt{1-4c}}, & B_n&=\frac{b}{1-4c},  &n\ge1, \\
  \Lambda_1&=\frac{\mathcal C(b^2c-4c\lambda+\lambda)}{\sqrt{1-4c}^3}, & \Lambda_{n+1}&=\frac{b^2c-4c\lambda+\lambda}{(1-4c)^2},&n\ge1,\\
  A_1 &= \frac{\mathcal C a}{\sqrt{1-4c}}, & A_{n+1} &= \frac{a}{1-4c}, & n\ge1.
\end{align*}
\end{proposition}
\begin{proof}
  The proof is similar to the proof of~\Cref{prop:R2_Classical}. By a
  direct computation, we obtain
  \begin{align*}
  \mathcal{C}  \sum_{n\ge0} \mu_n(\vec B,\vec \Lambda,\vec A) x^n
    &=\cfrac{\mathcal{C}}{1-B_0x-\cfrac{A_1x+\Lambda_1x^2}{1-B_1x-\cfrac{A_2x+\Lambda_2x^2}{1-B_2x-\cdots}}}\\
    &= \cfrac{\mathcal{C}}{1-B_0x-\cfrac{2A_1x+2\Lambda_1x^2}{1-B_1x+\sqrt{1-(4A_2+2B_1)x-(4\Lambda_2-B_1^2)x^2}}}\\
   &= \mu(x). \qedhere
  \end{align*}
\end{proof}

The moment \( \mu_n(\vec B, \vec \Lambda) \) of ordinary orthogonal
polynomials is the generating function for Motzkin paths from
\( (0,0) \) to \( (n,0) \). Likewise,
\( \mu_n(\vec B, \vec \Lambda,\vec A) \), the moment of orthogonal
polynomials of type \( R_I \), is the generating function for
Motzkin--Schr\"{o}der paths with the same endpoints. Both are finite
sums. Hence, Propositions~\ref{prop:R2_Classical} and
\ref{prop:R2_R1} provide practical methods for computing \( \mu_n \).
At this point, we can pose natural questions:
\begin{enumerate}
\item Are the path models in Propositions~\ref{prop:R2_Classical} and
  \ref{prop:R2_R1} applicable for \( \mu_{n,m} \) or
  \( \mu_{n,r,s} \)?
\item If not, are there Motzkin or Motzkin--Schr\"oder paths with
  different weights that can be used for computing \( \mu_{n,m} \) or
  \( \mu_{n,r,s} \)?
\end{enumerate}

The answer to the first question is ``no''. For instance, one can
verify that the coefficient \( [c]\mu_{2,1} \) of \( c \) in
\( \mu_{2,1} \) is equal to \( 10a+8b \). In contrast, let
\( \hat{\mu}_{2,1} \) be the weight sum of Motzkin paths
(resp.~Motzkin--Schr\"oder paths) from \( (0,0) \) to \( (2,1) \) with
the weights given in~\Cref{prop:R2_Classical}
(resp.~\Cref{prop:R2_R1}). Then \( [c]\hat{\mu}_{2,1} = 11a+6b \)
(resp.~\( [c]\hat{\mu}_{2,1} = 7a+6b \)).

Regarding the second question, we observe that \( \mu_{i,i} = \mu_{i,0,i} = {\mathcal C}^{i+1} \).
This implies that it is impossible to find ordinary or type \( R_I \) orthogonal polynomials whose generalized moments \( \widetilde{\mu}_{n,r,s} \) are the same as \( \mu_{n,r,s} \), since \( \widetilde{\mu}_{i,0,i} = 1 \).
One might consider introducing a scaling factor \( F \) such that \( \mu_{n,r,s} = F\cdot \widetilde{\mu}_{n,r,s} \) as we have \( F={\mathcal C} \) when \( r=s=0 \).
However, we note that it is not possible to interpret
\( \mu_{n,r,s} \) as a weight sum of Motzkin--Schr\"{o}der paths.
In fact, one can verify that
\begin{equation}\label{eq:22}
  \mu_{0,1,0} = \mathcal C^2\cdot\frac{a+2bc}{\sqrt{1-4c}} \quad\text{and}\quad
  \mu_{0,2,1} = \mu_{0,1,0}(2\mathcal C-1).
\end{equation}
To interpret these using Motzkin--Schr\"{o}der paths,
\( \mu_{0,i,i-1} \) must be the weight \( A_i \) of a vertical down
step \( (0,-1) \) starting at height \( i \).
This implies that
\( \mu_{0,2,0} = A_2A_1 = \mu_{0,2,1} \mu_{0,1,0} \), which
does not contain any \( \lambda \) by \eqref{eq:22}.
This contradicts the fact that \( \mu_{0,2,0} \) must contain a term
\( \lambda c \), since \( DB \) is an \( R_{II} \) path from
\( (0,2) \) to \( (0,0) \) whose weight is \( \lambda c \).

Therefore, we need to consider combinatorial models different from
Motzkin--Schr\"oder paths.
In the literature, there are other lattice path models
for orthogonal polynomials.
For example, Jang and Song \cite{Jang2024}
provide lattice paths called gentle Motzkin paths
for their study of orthogonal polynomials on the unit circle.

We end this section with the following open problem.
\begin{problem}
  Find a lattice path model that can be used to calculate
  \( \mu_{n,m} \) or \( \mu_{n,r,s} \) in a finite process.
\end{problem}

\end{document}